\documentclass[12pt,twoside]{article}
\usepackage{amsmath,amsfonts,amssymb,a4,enumitem,epsfig,array,psfrag}
\usepackage{url}
\usepackage{amsthm}
\usepackage{color}
\usepackage{calc}
\usepackage{etoolbox} 

\newtheorem{theo}{Theorem}
\newtheorem{prop}{Proposition}[section]
\newtheorem{coro}[prop]{Corollary}
\newtheorem{lemma}[prop]{Lemma}

\theoremstyle{definition}
\newtheorem{remark}[prop]{Remark}
\newtheorem{definition}[prop]{Definition}
\newtheorem{example}[prop]{Example}

\AtEndEnvironment{example}{\null\hfill$\diamond$}%
\AtEndEnvironment{remark}{\null\hfill$\diamond$}%
\AtEndEnvironment{definition}{\null\hfill$\diamond$}%

\newcommand{\bft}{{\mathbf t}}
\newcommand{\bff}{{\mathbf f}}
\newcommand{\bt}{{\mathbf t}}
\newcommand{\bM}{{\mathbf M}}
\newcommand{\bT}{{\mathbf T}}

\newcommand{\Aut}{\operatorname{Aut}}

\newcommand{\emptyplug}{{\mathbf p_\circ}}
\newcommand{\fullplug}{{\mathbf p_\bullet}}
\newcommand{\plug}{\operatorname{plug}}
\newcommand{\floorop}{\operatorname{floor}}
\newcommand{\Prob}{\operatorname{Prob}}

\newcommand{\nvert}{\operatorname{vert}}

\newcommand{\interior}{\operatorname{int}}

\newcommand{\flux}{\operatorname{flux}}

\newcommand{\sign}{\operatorname{sign}}

\newcommand{\NN}{{\mathbb{N}}}
\newcommand{\ZZ}{{\mathbb{Z}}}
\newcommand{\QQ}{{\mathbb{Q}}}
\newcommand{\RR}{{\mathbb{R}}}

\newcommand{\cT}{{\cal T}}
\newcommand{\cS}{{\cal S}}

\newcommand{\cC}{\mathcal{C}}
\newcommand{\cD}{{\cal D}}
\newcommand{\cP}{{\cal P}}
\newcommand{\cR}{{\cal R}}

\newcommand{\Hom}{\operatorname{Hom}}
\newcommand{\Tw}{\operatorname{Tw}}

\newcommand{\tw}{\operatorname{tw}}

\begin{document}
\title{Domino tilings of cylinders: the domino group \\
and connected components under flips}
\author{Nicolau C. Saldanha}

\maketitle

\begin{abstract}
We consider domino tilings of three-dimensional cubiculated regions.
In order to study such tilings, we define new objects:
the \textit{domino group} and \textit{domino complex}
of a quadriculated disk.
As an application, we study the problem of connectivity via flips.

A flip is a local move: two neighboring parallel dominoes
are removed and placed back in a different position.
The twist is an integer associated to each tiling,
which is invariant under flips.
A balanced quadriculated disk $\cD$ is {\em regular} if
whenever two tilings $\bt_0$ and $\bt_1$ of $\cD \times [0,N]$
have the same twist then
$\bt_0$ and $\bt_1$ can be joined by a sequence of flips
provided some extra vertical space is allowed.
We show that
$\cD$ is regular if and only if its domino group
is isomorphic to $\ZZ \oplus \ZZ/(2)$.
We prove that a rectangle $\cD = [0,L] \times [0,M]$
with $LM$ even is regular if and only if $\min\{L,M\} \ge 3$
and conjecture that in general ``large'' disks are regular.

We also prove that if $\cD$ is regular then
the extra vertical space necessary to join by flips
two tilings of $\cD \times [0,N]$ with the same twist
depends only on $\cD$, not on the height $N$.
Furthermore, almost every pair of tilings
with the same twist can be joined via flips
(with no extra space).

In the cases where $\cD$ is rectangular but not regular
we prove partial results
concerning the structure of the domino group:
the group is not abelian and has exponential growth.
Connected components via flips are now small
and almost no pair of tilings
with the same twist can be joined via flips.
\end{abstract}



\bigbreak

\section{Introduction}

Let $\cD \subset \RR^2$ be a quadriculated region
in the plane, i.e., a union of finitely many unit squares
$[a,a+1] \times [b,b+1]$, $(a,b) \in \ZZ^2$.
A {\em domino} is a closed $2\times 1$ or $1\times 2$ rectangle;
a {\em domino tiling} of $\cD$ is a covering of $\cD$
by dominoes with disjoint interiors.
In a more combinatorial language,
$\cD$ can be identified with a bipartite graph:
vertices of the graph are unit squares in $\cD$
and adjacent squares are joined by edges.
A domino tiling of $\cD$ is then a perfect matching;
we prefer to speak of dominoes and tilings (instead of edges and matchings).

A quadriculated region $\cD$ is a planar quadriculated {\em disk}
if $\cD$ is contractible with contractible interior
and therefore homeomorphic to the unit disk.
In particular, $\cD$ is connected and simply connected
with connected interior.
The color of a square is $(-1)^{a+b}$,
with $+1$ equal black and $-1$ equal white.
We always assume that our quadriculated regions $\cD$ are {\em balanced}
(equal number of white and black squares).
We sometimes assume that $\cD$ is {\em tileable}
(admits at least one domino tiling).
A quadriculated disk $\cD$ is {\em nontrivial}
if it has at least $6$ unit squares and
at least one square has at least three neighbours;
we usually also assume that our quadriculated disks are nontrivial.

\begin{figure}[h]
\begin{center}
\includegraphics[scale=0.275]{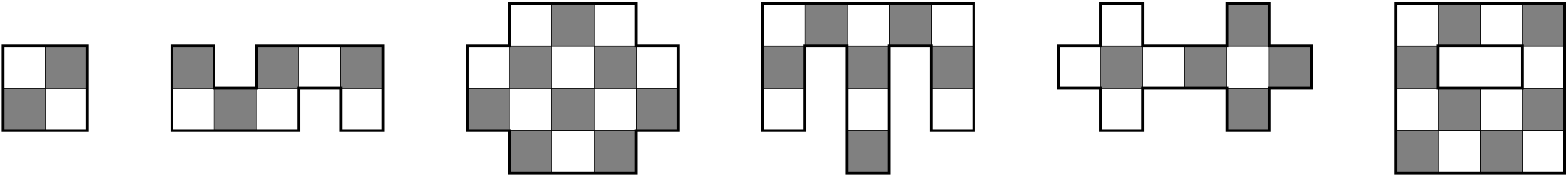}
\end{center}
\caption{Six examples of quadriculated regions, all connected and balanced.
The first two are trivial disks.
The third and fourth are nontrivial tileable quadriculated disks.
The fifth one is a nontrivial balanced quadriculated disk
which is not tileable. The sixth example is not a disk.}
\label{fig:disks}
\end{figure}

\begin{figure}[b!]
\centering
\def\svgwidth{75mm}
\begingroup%
  \makeatletter%
  \providecommand\color[2][]{%
    \errmessage{(Inkscape) Color is used for the text in Inkscape, but the package 'color.sty' is not loaded}%
    \renewcommand\color[2][]{}%
  }%
  \providecommand\transparent[1]{%
    \errmessage{(Inkscape) Transparency is used (non-zero) for the text in Inkscape, but the package 'transparent.sty' is not loaded}%
    \renewcommand\transparent[1]{}%
  }%
  \providecommand\rotatebox[2]{#2}%
  \newcommand*\fsize{\dimexpr\f@size pt\relax}%
  \newcommand*\lineheight[1]{\fontsize{\fsize}{#1\fsize}\selectfont}%
  \ifx\svgwidth\undefined%
    \setlength{\unitlength}{340.62662635bp}%
    \ifx\svgscale\undefined%
      \relax%
    \else%
      \setlength{\unitlength}{\unitlength * \real{\svgscale}}%
    \fi%
  \else%
    \setlength{\unitlength}{\svgwidth}%
  \fi%
  \global\let\svgwidth\undefined%
  \global\let\svgscale\undefined%
  \makeatother%
  \begin{picture}(1,0.67028813)%
    \lineheight{1}%
    \setlength\tabcolsep{0pt}%
    \put(0,0){\includegraphics[width=\unitlength,page=1]{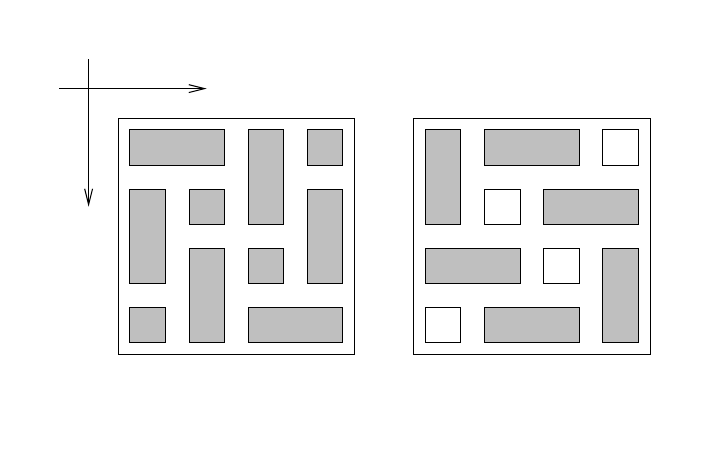}}%
    \put(0.33306534,0.54542844){\color[rgb]{0,0,0}\makebox(0,0)[lt]{\lineheight{1.25}\smash{\begin{tabular}[t]{l}$y$\end{tabular}}}}%
    \put(0.08321856,0.33722279){\color[rgb]{0,0,0}\makebox(0,0)[lt]{\lineheight{1.25}\smash{\begin{tabular}[t]{l}$x$\end{tabular}}}}%
    \put(0.24978308,0.08737601){\color[rgb]{0,0,0}\makebox(0,0)[lt]{\lineheight{1.25}\smash{\begin{tabular}[t]{l}$z\in[0,1]$\end{tabular}}}}%
    \put(0.66619438,0.08737601){\color[rgb]{0,0,0}\makebox(0,0)[lt]{\lineheight{1.25}\smash{\begin{tabular}[t]{l}$z\in[1,2]$\end{tabular}}}}%
  \end{picture}%
\endgroup%

\caption{A tiling of the box $[0,4]\times [0,4]\times [0,2]$.
The orientation of $\RR^3$ is very important:
the $z$ axes points away from the paper.
Examples of dominoes in this tiling are
$[0,1]\times[0,2]\times[0,1]$, $[0,2]\times[0,1]\times[1,2]$
and $[1,2]\times[1,2]\times[0,2]$.}
\label{fig:twist2}
\end{figure}
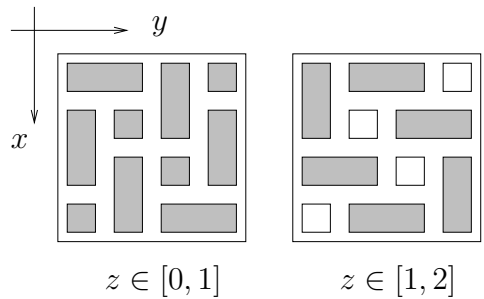

A cubiculated region is a set $\cR \subset \RR^3$ 
which is a union of finitely many unit cubes
$[a,a+1] \times [b,b+1] \times [c,c+1]$, $(a,b,c) \in \ZZ^3$.
The color of a cube is $(-1)^{a+b+c}$.
In this paper we always assume $\cR$ to be balanced and contractible
with contractible interior.
A cylinder is a simple example of a cubiculated region:
\[ \cR_N = \cD \times [0,N] \subset \RR^3, \]
where $\cD$ is a fixed balanced quadriculated disk.
A box is a special case of a cylinder:
$\cD = [0,L] \times [0,M]$,  $LM$ even;
$\cR_N = \cD \times [0,N]$.
A (3D) domino is the union of two unit cubes with a common face,
thus a $2\times 1\times 1$ rectangular cuboid.
A (domino) tiling of $\cR$ is a family of dominoes with disjoint interiors
whose union is $\cR$.
Again, $\cR$ can be identified with a bipartite graph,
dominoes with edges and tilings with matchings.
The set of domino tilings of a region $\cR$ is denoted by $\cT(\cR)$.

We follow \cite{primeiroartigo} and \cite{segundoartigo}
in drawing tilings of cubiculated regions by floors,
as in Figure \ref{fig:twist2}.
Vertical dominoes
(i.e., dominoes in the $z$ direction)
appear as two squares, one in each of two adjacent floors;
for visual facility, we leave the right square unfilled.

A flip is a local move in $\cT(\cR)$:
two parallel and adjacent (3D) dominoes are removed
and placed back in a different position.
Examples of flips are shown in Figure \ref{fig:flipexample}.
This is of course a natural generalization of the planar flip.
It is well known that if $\cD \subset \RR^2$
is a quadriculated disk
then any two tilings of $\cD$ can be joined by a finite sequence of flips
(\cite{thurston1990}, \cite{saldanhatomei1995}).
This is not true for 3D regions and tilings.
For $\bt_0, \bt_1 \in \cT(\cR)$, we write $\bt_0 \approx \bt_1$
if $\bt_0$ and $\bt_1$ can be joined by a finite sequence of flips.

\begin{figure}[ht]
\begin{center}
\includegraphics[scale=0.275]{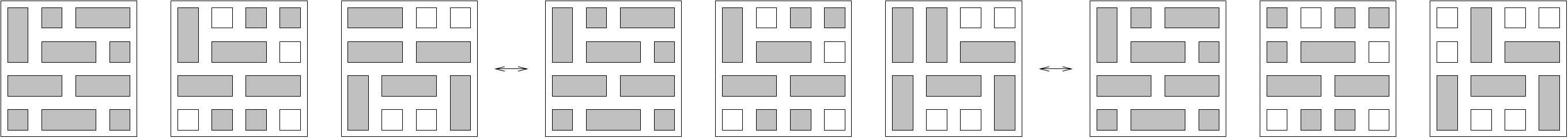}
\end{center}
\caption{Three tilings of the box $[0,4]\times[0,4]\times[0,3]$
and two flips.}
\label{fig:flipexample}
\end{figure}

A trit is the only local move involving three dominoes
which does not reduce to flips
(see \cite{primeiroartigo}, \cite{segundoartigo}, \cite{FKMS}).
The three dominoes involved are in three different directions
and fill a $2\times 2\times 2$ box minus two opposite unit cubes.
Figure \ref{fig:trit} shows a trit in the $3\times 3\times 2$ box;
notice that the first tiling admits no flips.

\begin{figure}[h]
\begin{center}
\includegraphics[scale=0.275]{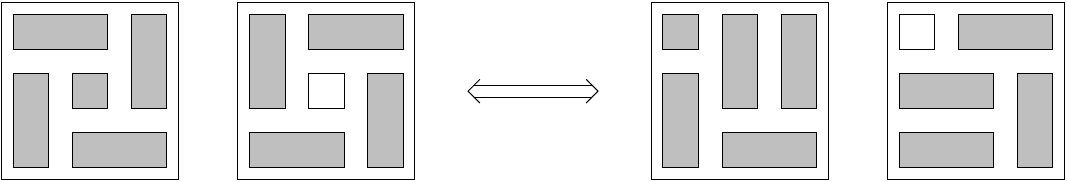}
\end{center}
\caption{Two tilings of the box $[0,3]\times[0,3]\times[0,2]$
joined by a trit.
The first tiling has twist $-1$, the second one has twist $0$.}
\label{fig:trit}
\end{figure}

For a fixed balanced quadriculated disk $\cD$,
let $\bt_0 \in \cT(\cR_{N_0})$ and $\bt_1 \in \cT(\cR_{N_1})$.
These tilings can be concatenated to define a tiling
$\bt_0 \ast \bt_1 \in \cT(\cR_{N_0+N_1})$.
If $\bt_0$ and $\bt_1$ are drawn as in our figures
(say, Figures \ref{fig:flipexample} and \ref{fig:trit}),
then the figure for $\bt_0 \ast \bt_1$ is obtained by concatenating
the figures for $\bt_0$ and $\bt_1$.
Equivalently, we may translate $\bt_1$ by $(0,0,N_0)$ to obtain
a tiling $\tilde\bt_1$ of $\cD \times [N_0,N_0+N_1]$.
The set of dominoes forming $\bt_0 \ast \bt_1$ 
is the disjoint union of the set of dominoes forming $\bt_0$ and $\tilde\bt_1$.

For $N$ even, there exists a tiling $\bt_{\nvert,N} \in \cT(\cR_N)$
such that all dominoes are vertical
(i.e., of the form $[a,a+1]\times [b,b+1] \times [c,c+2]$);
we call this the vertical tiling.
For $\cD = [0,4]^2$,
let $\bt_{\nvert,2}, \bt_0, \bt_1 \in \cT(\cR_2)$
be the tilings in Figure \ref{fig:442}.
Clearly, $\bt_0 \not\approx \bt_1$ (neither admits a flip);
a computation verifies that 
$\bt_0 \ast \bt_{\nvert,2} \approx \bt_1 \ast \bt_{\nvert,2}$.


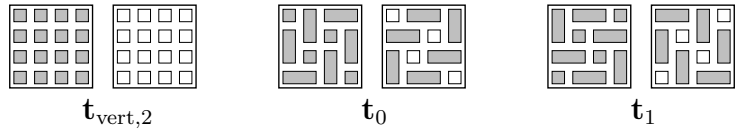
\begin{figure}[h]
\centering
\def\svgscale{0.275}
\begingroup%
  \makeatletter%
  \providecommand\color[2][]{%
    \errmessage{(Inkscape) Color is used for the text in Inkscape, but the package 'color.sty' is not loaded}%
    \renewcommand\color[2][]{}%
  }%
  \providecommand\transparent[1]{%
    \errmessage{(Inkscape) Transparency is used (non-zero) for the text in Inkscape, but the package 'transparent.sty' is not loaded}%
    \renewcommand\transparent[1]{}%
  }%
  \providecommand\rotatebox[2]{#2}%
  \newcommand*\fsize{\dimexpr\f@size pt\relax}%
  \newcommand*\lineheight[1]{\fontsize{\fsize}{#1\fsize}\selectfont}%
  \ifx\svgwidth\undefined%
    \setlength{\unitlength}{1051.42359791bp}%
    \ifx\svgscale\undefined%
      \relax%
    \else%
      \setlength{\unitlength}{\unitlength * \real{\svgscale}}%
    \fi%
  \else%
    \setlength{\unitlength}{\svgwidth}%
  \fi%
  \global\let\svgwidth\undefined%
  \global\let\svgscale\undefined%
  \makeatother%
  \begin{picture}(1,0.20476916)%
    \lineheight{1}%
    \setlength\tabcolsep{0pt}%
    \put(0,0){\includegraphics[width=\unitlength,page=1]{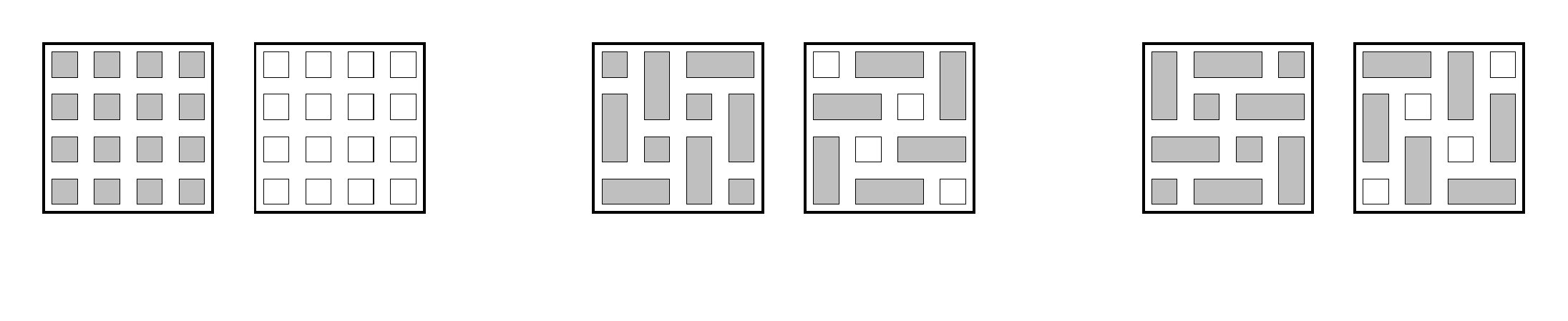}}%
    \put(0.12228751,0.02852272){\color[rgb]{0,0,0}\makebox(0,0)[lt]{\lineheight{1.25}\smash{\begin{tabular}[t]{l}$\bt_{\nvert,2}$\end{tabular}}}}%
    \put(0.48651027,0.02852272){\color[rgb]{0,0,0}\makebox(0,0)[lt]{\lineheight{1.25}\smash{\begin{tabular}[t]{l}$\bt_0$\end{tabular}}}}%
    \put(0.83724329,0.02852272){\color[rgb]{0,0,0}\makebox(0,0)[lt]{\lineheight{1.25}\smash{\begin{tabular}[t]{l}$\bt_1$\end{tabular}}}}%
  \end{picture}%
\endgroup%

\caption{Three tilings $\bt_{\nvert,2}$, $\bt_0$ and $\bt_1$
of the box $[0,4]\times[0,4]\times[0,2]$.
We have $\bt_0 \not\approx \bt_1$ but $\bt_0 \sim \bt_1$.
Both $\bt_0$ and $\bt_1$ have twist $+2$. Neither admits a flip.}
\label{fig:442}
\end{figure}

Motivated by this example,
we define a weaker equivalence relation $\sim$ on tilings
(meaning that $\bt_0 \approx \bt_1$ always implies $\bt_0 \sim \bt_1$).
Assume $N_0 \equiv N_1 \pmod 2$ and $\bt_i \in \cT(\cR_{N_i})$:
$\bt_0 \sim \bt_1$ if and only if there exist $M_0 \in 2\NN$ and
$M_1 = N_0 + M_0 - N_1 \in 2\NN$ such that
$\bt_0 \ast \bt_{\nvert,M_0} \approx \bt_1 \ast \bt_{\nvert,M_1}$.
Thus, for the tilings in Figure \ref{fig:442} we have $\bt_0 \sim \bt_1$
(in this case, we can take $M_0 = M_1 = 2$).

In \cite{FKMS} the concept of a {\em refinement} of a tiling is introduced.
Once we are familiar with the concept of refinement,
it is not hard to see that if $\bt_0, \bt_1 \in \cT(\cR_n)$ satisfy
$\bt_0 \sim \bt_1$ then there exist refinements
$\tilde\bt_0, \tilde\bt_1$ with $\tilde\bt_0 \approx \tilde\bt_1$
but the converse is not always true.
We shall not require the concept of refinement in the present paper.

\bigskip

Given a balanced quadriculated disk $\cD$,
we define the {\em full domino group $G_{\cD}$}
and the {\em even domino group $G^{+}_{\cD}$},
a subgroup of index $2$ of $G_{\cD}$.
The set of elements of the group is
the set of equivalence classes of $\sim$;
for $G^{+}_{\cD}$ we only take even values of $N$:
\begin{equation}
\label{eq:dominogroup}
G_{\cD} = \left( \bigsqcup_{N \in \NN^\ast} \cT(\cR_{N}) \right)/\sim
\quad > \quad
G^{+}_{\cD} = \left( \bigsqcup_{N \in \NN^\ast} \cT(\cR_{2N}) \right)/\sim.
\end{equation}
As usual with spaces defined as a quotient under an equivalence relation,
we abuse notation and think of the elements of $G_{\cD}$ as tilings $\bt$.
The operation in $G_{\cD}$ is $\ast$, the concatenation;
the identity element is $\bt_{\nvert,2}$.
The inverse of a tiling $\bt \in \cT(\cR_N)$ is
$\bt^{-1} \in \cT(\cR_N)$
obtained by reflecting in the $z$ coordinate.
There is a homomorphism $G_{\cD} \to \{\pm 1\}$
taking $\bt \in \cT(\cR_N)$ to $N \bmod 2$;
by construction, $G^{+}_{\cD}$ is the kernel of this homomorphism.
As we shall see in Section \ref{sect:plug},
$G_{\cD}$ is a finitely presented group,
the fundamental group of an explicit finite complex.
The domino group $G_{\cD}$ has different structures
for different quadriculated disks $\cD$
and it indicates the behavior of connected components under flips
of the region $\cR_N$, particularly for large values of $N$.

The {\em twist}
(see \cite{primeiroartigo}, \cite{segundoartigo}, \cite{FKMS})
is a group homomorphism $\Tw: G_{\cD} \to \ZZ$.
A valid definition of the twist for tilings of cylinders
is presented in Section \ref{sect:twist}.
(There are many equivalent definitions of twist:
the one presented in \cite{FKMS} is very general
and uses homology theory;
the one presented in \cite{primeiroartigo} is more elementary
but still complicated, and works only for certain regions.
Unfortunately, none of these definitions is as simple as might be desired.)
We recall a few basic facts:
trits change the value of the twist by adding $\pm 1$.
For cylinders, taking a mirror image of a tiling
changes the sign of its twist:
in particular, $\Tw(\bt^{-1}) = -\Tw(\bt)$.
On the other hand,
if a rotation preserves $\cR$
then it also preserves twist.
If $\cD$ is not trivial then the map $\Tw$ is surjective;
in particular, $G_{\cD}$ is infinite.
A quadriculated disk $\cD$ is {\em regular}
if  $\Tw: G^{+}_{\cD} \to \ZZ$ is an isomorphism,
which implies $G_{\cD} \approx \ZZ \oplus (\ZZ/(2))$.
We conjecture that a quadriculated disk is regular
unless it has a norrow bottleneck.
The following theorem is a special case of this conjecture;
more evidence in favor of the conjecture can be found in \cite{marreiros}.

\begin{theo}
\label{theo:rectangle}
Let $\cD = [0,L] \times [0,M]$ be a rectangle with $LM$ even.
Then $\cD$ is regular if and only if $\min\{L,M\} \ge 3$.
\end{theo}

\begin{figure}[h]
\begin{center}
\includegraphics[scale=0.275]{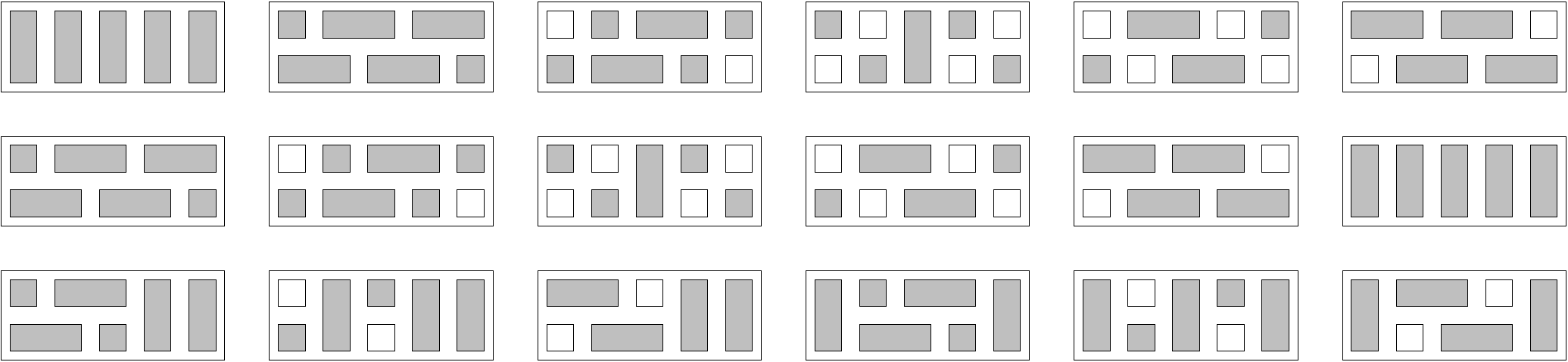}
\end{center}
\caption{Three tilings $\bt_0$, $\bt_1$ and $\bt_2$ of $\cR_6$
for $\cD = [0,2]\times[0,5]$.
We have $\bt_0 \not\sim \bt_1 \not\sim \bt_2 \not\sim \bt_0$ and
$\Tw(\bt_0) = \Tw(\bt_1) = \Tw(\bt_2)$.}
\label{fig:L2}
\end{figure}

As we shall see Lemma \ref{lemma:thin} (in Section \ref{sect:thin}),
for $\cD = [0,2] \times [0,M]$, $M \ge 3$,
there exists a surjective map
$\phi: G_{\cD} \to F_2 \ltimes \ZZ/(2)$
where $F_2$ is the free group with two generators.
This implies that $G_{\cD}$ has exponential growth
and therefore
$\Tw(\bt_0) = \Tw(\bt_1)$ is far from implying that $\bt_0 \sim \bt_1$.
For instance,
the three tilings $\bt_0$, $\bt_1$ and $\bt_2$ shown in Figure \ref{fig:L2}
satisfy $\Tw(\bt_0) = \Tw(\bt_1) = \Tw(\bt_2)$ and
$\bt_0 \not\sim \bt_1 \not\sim \bt_2 \not\sim \bt_0$
(this claim and the similar one in Figure \ref{fig:234}
will follow from applying the map $\phi$,
constructed in Lemma \ref{lemma:thin}).

By definition, if $\cD$ is a regular quadriculated disk
and $\bt_0, \bt_1 \in \cT(\cR_N)$ satisfy $\Tw(\bt_0) = \Tw(\bt_1)$
then there exists $M \in 2\NN$ such that
$\bt_0 \ast \bt_{\nvert,M} \approx \bt_1 \ast \bt_{\nvert,M}$.
It is natural to ask about the size of $M$.


\begin{theo}
\label{theo:M}
Let $\cD$ be a regular quadriculated disk containing a $2\times 3$ rectangle.
Then there exists $M$ (depending on $\cD$ only)
such that for all $N \in \NN$ and for all $\bt_0, \bt_1 \in \cT(\cR_N)$
if $\Tw(\bt_0) = \Tw(\bt_1)$
then $\bt_0 \ast \bt_{\nvert,M} \approx \bt_1 \ast \bt_{\nvert,M}$.
\end{theo}

There are several ways in which it would be desireable to improve this result,
such as providing an estimate for $M$
(following the proof gives us at best a crude estimate for $M$,
and the computations seem daunting even for small examples).
We do not try to state or prove a related result for irregular disks;
the proof of Theorem \ref{theo:M} suggests that we should decide
whether the domino group $G_{\cD}$ is hyperbolic
(see Remark \ref{remark:hyperbolic}).

The consequences of regularity are further explored in other papers;
here we provide a sample.
Given a cubiculated region $\cR$, 
we consider random tilings $\bT$ of $\cR$
(i.e., random variables $\bT: \Omega \to \cT(\cR)$).
Such random tilings are assumed to be chosen uniformly in $\cT(\cR)$.
There are significant difficulties
in implementing such random variables
(with practical computer programs
and correct probability distribution).
In a related vein,
even giving a good estimate for $|\cT(\cR)|$
(where $\cR$ a large box)
is an open problem.

We describe a probabilistic result
which is stated and proved in \cite{saldanhaejc}.
Consider the cardinality
$|\ker(\Tw)| \in \NN^\ast \cup \{\infty\}$
of the normal subgroup $\ker(\Tw) < G^{+}_{\cD}$,
the kernel of $\Tw: G^{+}_{\cD} \to \ZZ$.
If $\ker(\Tw)$ is infinite,
$1/|\ker(\Tw)|$ is understood to be equal to $0$.
For non trivial $\cD$,
$|\ker(\Tw)| = 1$ if and only if $\cD$ is regular.
For $\cD = [0,L] \times [0,M]$ (with $LM$ even),
Theorem \ref{theo:rectangle}
and Lemma \ref{lemma:thin} imply that:
\[ \frac{1}{|\ker(\Tw)|} = \begin{cases}
1, & \min\{L,M\} \ge 3, \\ 0, & \min\{L,M\} = 2. \end{cases} \]
There are no known examples of irregular disks for which 
$|\ker(\Tw)| \notin \{1,\infty\}$.

Let $\cD$ be a non trivial quadriculated disk.
Let $G^{+}_{\cD}$ be the even domino group;
let $\Tw: G^{+}_{\cD} \to \ZZ$ be the twist map.
Let $\bT_0, \bT_1$ be independent random tilings
of $\cR_N = \cD \times [0,N]$;
we have
\begin{equation}
\label{equation:limprob}
\lim_{N \to \infty} \Prob[\bT_0 \approx \bT_1 | \Tw(\bT_0) = \Tw(\bT_1)]
= \frac{1}{|\ker(\Tw)|}.
\end{equation}
Equation \ref{equation:limprob} is essentially
Theorem 6 from \cite{saldanhaejc}.
In particular, the probability tends to $1$
if and only if $\cD$ is regular
and tends to $0$ if $\cD = [0,2] \times [0,M]$.

\smallskip

Section \ref{sect:examples} lists a few computational examples.
In Sections \ref{sect:plug} and \ref{sect:cork} we present
some helpful concepts, such as that of a {\em plug}, a {\em floor}
and a {\em cork}; we also construct certain tilings which 
will be important again later.
In Section \ref{sect:groupcomplex} we construct a $2$-complex,
the {\em domino complex} $\cC_{\cD}$
(where $\cD$ is a balanced quadriculated disk):
tilings of $\cR_N = \cD \times [0,N]$ correspond to closed
paths of length $N$ in $\cC_{\cD}$
and the domino group $G_{\cD}$ is the fundamental group $\pi_1(\cC_{\cD})$.
In Section \ref{sect:twist} we present a self-contained definition
of twist and prove some basic properties.
Section \ref{sect:thin} is dedicated to proving that rectangles
$[0,2] \times [0,M]$ (for $M \ge 3$) are {\em not} regular.
The fact that $G_{\cD}$ is the fundamental group of a finite $2$-complex
implies that it is finitely presented,
and gives us an explicit finite family of generators.
This family is finite, but too large to be useful
in explicit computations:
in Section \ref{sect:gen} we present a far smaller,
and therefore more manageable, family of generators.
Given such a family, it is not hard to produce an algorithm which,
given an explicit disk $\cD$, will,
if $\cD$ is regular, produce a proof of this fact in finite time;
if $\cD$ is not regular the algorithm will run forever.
For small examples, the algorithm can actually be executed.
In Section \ref{sect:44} we walk through this algorithm
for $\cD = [0,4]\times[0,4]$, proving that it is regular;
this is Lemma \ref{lemma:44}.
Similarly, we see in Lemma \ref{lemma:34} that
if $L,M \in [3,6] \cap \ZZ$ and $LM$ is even then
$\cD = [0,L]\times [0,M]$ is regular.
In Section \ref{sect:thick} we prove
Theorem \ref{theo:rectangle}. 
In Section \ref{sect:cD} we define
the constant $c_{\cD} \in \QQ \cap (0,+\infty)$
for a regular disk $\cD$ and construct
a quasi-isometry between $\tilde\cC_{\cD}$
and the {\em spine},
a subcomplex $\tilde\cC^{\bullet}_{\cD} \subset \tilde\cC_{\cD}$
isometric to $\RR$, with vertices in $\ZZ$.
In Section \ref{sect:theoM} we prove Theorem \ref{theo:M}. 
Finally, Section \ref{sect:final} contains a few final remarks.

\smallskip

The author thanks Juliana Freire, Caroline Klivans,
Raphael de Marreiros, Pedro Milet and Breno Pereira for helpful conversations,
comments and suggestions;
Caroline Klivans read a preliminary version and provided 
detailed and helpful recommendations.
The author thanks the referee
for several insightful and productive suggestions.
The author is also thankful for the generous support of
CNPq, CAPES and FAPERJ (Brazil).


\section{Examples}
\label{sect:examples}

For the $3\times 3\times 2$ box
there exist $229$ tilings:
exactly one tiling has twist $-1$
(the first tiling shown in Figure \ref{fig:trit})
and exactly one tiling has twist $+1$ (its mirror image),
the other $227$ tilings have twist $0$ and
form an equivalence class under $\approx$ (and also under $\sim$).
For the $4\times 4\times 2$ box
there exist tilings with the same twist
but which can not be joined by a sequence of flips.
Indeed, for this region $\cR$, there are $32000$ tilings,
$5$ possible values for the twist (from $-2$ to $+2$)
and $\cT(\cR)$ has $9$ connected components via flips.
All $31484$ tilings with twist $0$ are in the same connected component.
The $256$ tilings with twist $+1$ form two connected components
with $128$ tilings each (and similarly for twist $-1$).
The two tilings in Figure \ref{fig:442} are the only tilings
with twist $+2$: notice that neither admits a flip
(similarly, there are two tilings with twist $-2$).
In this example, one can check that
$\Tw(\bt_0) = \Tw(\bt_1)$ implies $\bt_0 \sim \bt_1$.

For the $4\times 4\times 4$ box
there are $5051532105$ tilings,
$9$ possible values for the twist (from $-4$ to $+4$)
and the set of tilings $\cT(\cR)$ has $93$ equivalent classes under $\approx$.
The number of tilings of twist $0$ is $4413212553$,
forming one giant connected component with $4412646453$ tilings,
two components with $283044$ tilings each
and $12$ isolated tilings.
The number of tilings of twist $1$ is $310188792$,
forming one giant component with $310185960$ tilings
and $12$ components with $236$ tilings each.
In this example a brute force computation verifies that
$\Tw(\bt_0) = \Tw(\bt_1)$ implies
$\bt_0 \ast \bt_{\nvert,2} \approx \bt_1 \ast \bt_{\nvert,2}$
(and therefore $\bt_0 \sim \bt_1$).
In other words, if $\Tw(\bt_0) = \Tw(\bt_1)$ then,
after adding two floors of vertical dominoes,
the two resulting tilings are flip connected
(compare with Theorem \ref{theo:M}).
For the $4\times 4\times N$ box, $N$ a multiple of $4$,
the possible values of the twist are $[-\frac32 N+2,\frac32 N-2] \cap \ZZ$
(see Figure \ref{fig:rocket} and Example \ref{example:cD44}).

\begin{figure}[ht]
\begin{center}
\includegraphics[scale=0.275]{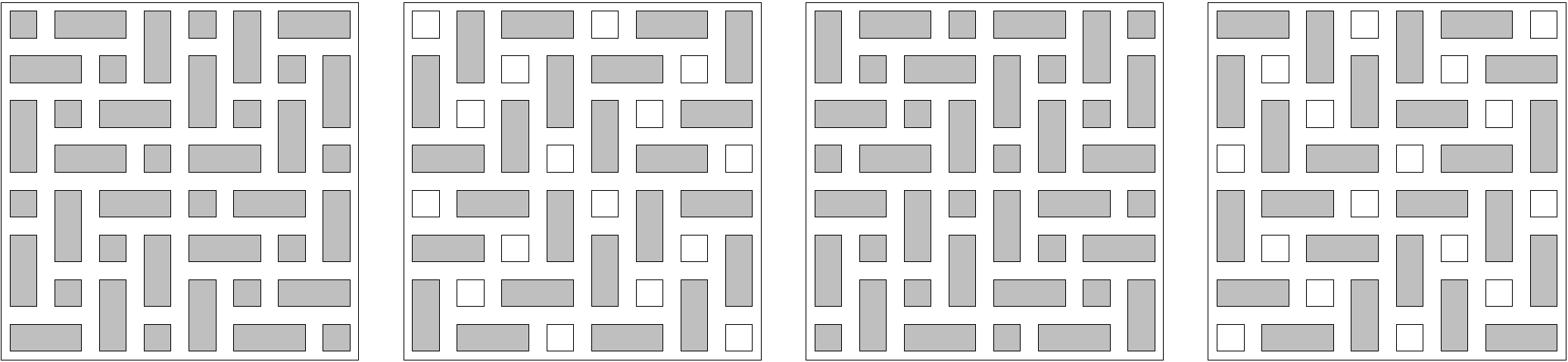}
\end{center}
\caption{This tiling $\bt_1$ satisfies $\Tw(\bt_1) = 0$ and admits no flips.}
\label{fig:noflip8}
\end{figure}

It is not easy to extend these computations to larger boxes
but there are examples of isolated tilings
(i.e., with no flips),
including tilings of twist $0$.
Figure \ref{fig:noflip8} shows an example of a tiling $\bt_1$
of $\cR_4$ for $\cD = [0,8]^2$;
the tiling $\bt_1$ admits no flips and satisfies $\Tw(\bt_1) = 0$.
If $\bt_2$ is obtained from $\bt_1$ by rotating $90^{\circ}$
in $xy$ then $\Tw(\bt_{\nvert,4}) = \Tw(\bt_1) = \Tw(\bt_2) = 0$ and
$\bt_{\nvert,4} \not\approx \bt_1
\not\approx \bt_2 \not\approx \bt_{\nvert,4}$
(since both $\bt_1$ and $\bt_2$ are isolated);
it is not hard to verify (by brute force) that
$\bt_{\nvert,4} \sim \bt_1 \sim \bt_2 \sim \bt_{\nvert,4}$.

As proved in \cite{saldanhaejc}, the number of tilings of $\cR_N$
per value of the twist approaches a normal distribution
when the basis $\cD$ is kept fixed and $N$ goes to infinity.
This may seem surprising given the results 
for the cube $4\times 4\times 4$,
but this region is too small for us to see the patern.
The normal distribution is proved in Theorem 4 of \cite{saldanhaejc};
it is also visible in Figure 2 of the same paper.




\section{Floors and plugs}
\label{sect:plug}

In this section, let $\cD \subset \RR^2$ be a fixed but arbitrary
balanced quadriculated disk
(thus $\cD$ is connected and simply connected, with connected interior).
Recall that $\cD$ is nontrivial if at least one unit square
has at least $3$ neighbors (see Figure \ref{fig:disks}).
Let $|\cD|$ be the number of squares of $\cD$
(so that $|\cD|$ is even;
if $\cD$ is nontrivial then $|\cD| \ge 6$).
Rectangles provide us with
a family of examples:
$\cD = [0,L] \times [0,M]$ where
$L, M \in \NN$, $2 \le L \le M$, and $LM$ is even.

For a disk $\cD$, a \emph{plug} is
a balanced quadriculated subregion $p \subseteq \cD$.
In other words, a plug $p$ is a union of finitely many unit squares
$[a,a+1] \times [b,b+1] \subset \cD$ (with $(a,b) \in \ZZ^2$)
such that the numbers of black and white squares in $p$ are equal.
From a graph point of view, $p$ is a balanced induced subgraph of $\cD$.
Let $\cP$ be the set of \emph{plugs} for $\cD$.
We allow for the \emph{empty plug} $\emptyplug = \emptyset \in \cP$
and the \emph{full plug} $\fullplug = \cD \in \cP$.
Each plug $p \in \cP$ has a complement $p^c \in \cP$:
the interiors of $p$ and $p^c$ are disjoint and $p \cup p^c = \cD$;
thus, for instance, $\fullplug = \emptyplug^c$.
We have $|\cP| = \binom{2k}{k}$, $k = |\cD|/2$.

Given plugs $p, \tilde p \in \cP$ and $N_0, N_1 \in \ZZ$, $N_1 > N_0 + 2$,
we define the {\em cork} $\cR_{N_0,N_1;p,\tilde p}$ 
to be the balanced cubiculated region
\[ \cR_{N_0,N_1;p,\tilde p} =
(\cD \times [N_0,N_1]) \smallsetminus
\interior((p \times [N_0,N_0+1]) \cup (\tilde p \times [N_1-1,N_1])). \]
Thus, $\cR_{N_0,N_1;p,\tilde p}$ is obtained from
$\cR_{N_0,N_1} = \cD \times [N_0,N_1]$ by removing
$p_0$ from the bottom floor $N_0+1$ (i.e., $\cD \times [N_0,N_0+1]$)
and $p_1$ from the top floor $N_1$.
Notice that $\cR_{0,N;\emptyplug,\emptyplug} = \cR_N$
and $\cR_{0,N;\fullplug,\fullplug}$ is a translated copy of $\cR_{N-2}$.

For disjoint $p, \tilde p \in \cP$,
consider the planar region
$\cD_{p, \tilde p} = \cD \smallsetminus (p \sqcup \tilde p)$
(we make here the usual abuse of neglecting boundaries).
The region $\cD_{p, \tilde p}$ is balanced,
but possibly neither connected not tileable;
we may also have $\cD_{p, \tilde p}$ empty.

Consider $p_{N_0}, p_{N_1} \in \cP$ and the
cork $\cR = \cR_{N_0,N_1; p_{N_0}, p_{N_1}}$.
A tiling $\bft \in \cT(\cR)$
can be described as a sequence
of \emph{floors} and plugs:
\begin{equation}
\label{equation:floorsandplugs}
(p_{N_0},f_{N_0+1},p_{N_0+1},f_{N_0+2},p_{N_0+2},\ldots
,p_{N_1-1},f_{N_1},p_{N_1}).
\end{equation}
The $j$-th plug $p_j = \plug_j(\bt) \in \cP$
is the union of the unit squares $[a,a+1]\times[b,b+1] \times \{j\}$
contained in $\cD \times \{j\}$ and crossed by a vertical domino
$[a,a+1]\times [b,b+1] \times [j-1,j+1]$ in the tiling $\bt$.
Similarly, the reduced $j$-th floor
$f^{\ast}_j = \floorop^{\ast}_j(\bt) \in \cT(\cD_{p_{j-1},p_{j}})$
corresponds to the set of horizontal dominoes of $\bft$ contained
in $\cD \times [j-1,j]$.
Notice that $p_{j-1}$ and $p_{j}$ are disjoint (for all $j$).
Figure \ref{fig:floorplug} shows a tiling represented as a sequence
of reduced floors and plugs.
The (full) $j$-th floor is
$f_j = \floorop_j(\bt) = (p_{j-1},f^{\ast}_j,p_j)$,
so that the representation in Equation \ref{equation:floorsandplugs}
is redundant.
Figures
\ref{fig:twist2}, \ref{fig:flipexample}, \ref{fig:trit}, \ref{fig:442}
(and others)
show examples of tilings represented as sequences of (full) floors.

\begin{figure}[h]
\begin{center}
\includegraphics[scale=0.275]{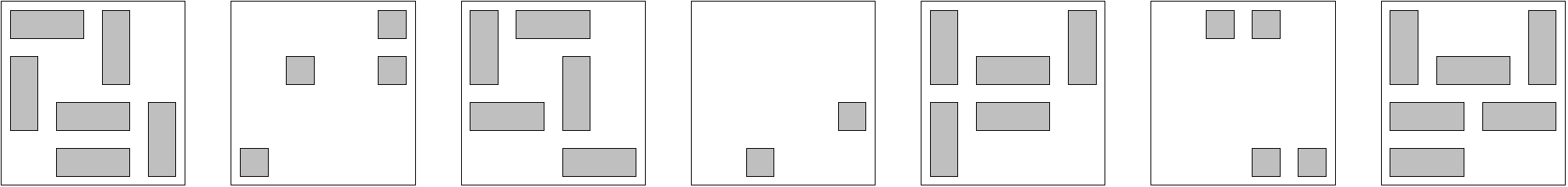}
\end{center}
\caption{A tiling of the box $4\times 4\times 4$
as a sequence of reduced floors and plugs.}
\label{fig:floorplug}
\end{figure}

Given a floor $f = (p_0,f^\ast,p_1)$,
define $f^{-1} = (p_1,f^\ast,p_0)$, also a valid floor;
we say that $f$ and $f^{-1}$ differ by orientation only
(see also Section \ref{sect:groupcomplex}).
A floor $f = (p_0,f^\ast,p_1)$ is \emph{vertical}
if $f^\ast = \emptyset$
(the empty tiling of the empty region $\cD_{p_0,p_1}$),
or, equivalently, if $p_1 = p_0^c$.
From the point of view of tilings,
a floor is vertical if all dominoes intersecting it are vertical.

\begin{lemma}
\label{lemma:trivialdisk}
If $\cD$ is a trivial quadriculated disk and 
$\bt_0, \bt_1 \in \cT(\cR_N)$ then
$\bt_0 \approx \bt_1$.
\end{lemma}

Recall that a disk $\cD$ is non trivial if at least one square
has at least three neighbors (see Figure \ref{fig:disks}).

\begin{proof}
If $\cD$ is not the $2\times 2$ square
then, as a graph, $\cD$ is isomorphic to $[0,1] \times [0,M]$
and therefore, as a graph,
$\cR_N$ is isomorphic to the quadriculated disk $[0,M] \times [0,N]$.
In other words, we have two interpretations of the same graph,
one of them 2D, the other 3D.
The meaning of flips is the same in both interpretations.
We then know from \cite{thurston1990} and \cite{saldanhatomei1995}
that $\bt_0 \approx \bt_1$.

If $\cD = [0,2]\times [0,2]$ then there are $6$ plugs:
in all cases, vertical dominoes can be matched in adjacent pairs.
Thus, a few vertical flips take any tiling to a tiling
with no vertical dominoes and then to a base tiling.
\end{proof}


\section{Tilings of corks}
\label{sect:cork}

This section is about Lemma \ref{lemma:cork},
which states that, for sufficiently large $N$,
the corks $\cR_{0,N;\emptyplug,p}$, $\cR_{0,N;p,\emptyplug}$ admit tilings.
Figure \ref{fig:cork} shows an example.


\begin{figure}[h]
\begin{center}
\includegraphics[scale=0.275]{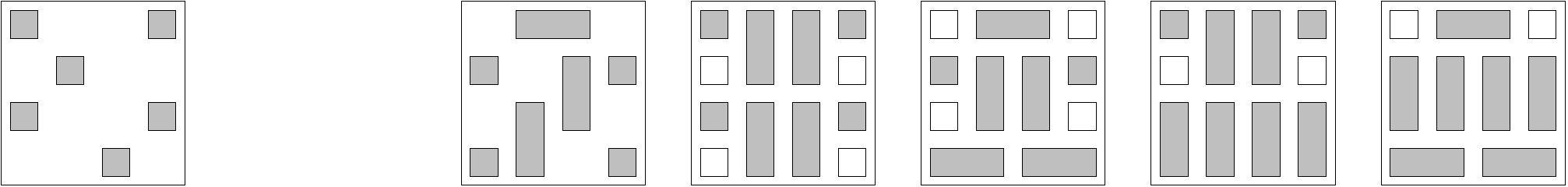}
\end{center}
\caption{A plug $p$ and a tiling of the cork $\cR_{5;p,\emptyplug}$.}
\label{fig:cork}
\end{figure}





Given $p \in \cP$, let $|p|$ be the number of squares in $p$.
If $N_1 - N_0$ is even and $p \in \cP$ 
then there exists a unique tiling
$\bt_{\nvert} \in \cT(\cR_{N_0,N_1;p,p})$ 
such that all floors are vertical.

For a balanced quadriculated disk $\cD$
and a plug $p \in \cP$,
consider the cork $\cR_{-N,N;p,p}$;
a tiling $\bt \in \cT(\cR_{-N,N;p,p})$
is {\em even} if $\bt$ is of the form:
\begin{equation}
\label{eq:bt0}
\bt = (p,f^{\ast}_{N},p_{N-1},f^{\ast}_{N-1}, \ldots,
p_1,f^{\ast}_1,p_0,f^{\ast}_1,p_1,\ldots,
f^{\ast}_{N-1},p_{N-1},f^{\ast}_N,p),
\end{equation}
i.e., if the tiling is symmetric with respect to
the reflection on the $xy$ plane.
An example of an even tiling is the vertical tiling
$\bt_{\nvert} \in \cT(\cR_{-N,N;p,p})$.

\begin{lemma}
\label{lemma:eventiling}
If a tiling $\bt \in \cT(\cR_{-N,N;p,p})$ is even then
$\bt \approx \bt_{\nvert}$.
\end{lemma}

\begin{proof}
Begin with $\bt_0 = \bt$ as in Equation \ref{eq:bt0} above.
Performing one vertical flip for each (horizontal) domino in $f_1^{\ast}$
takes us to
\[ \bt_1 = (p,f^{\ast}_{N},p_{N-1},f^{\ast}_{N-1}, \ldots,
p_2,f^{\ast}_2,p_1,\emptyset,p_1^c,\emptyset,p_1,f^{\ast}_2,p_2,\ldots,
f^{\ast}_{N-1},p_{N-1},f^{\ast}_N,p); \]
performing three vertical flips for each domino in $f_2^{\ast}$
then takes us to
\[ \bt_2 = (p,f^{\ast}_{N},p_{N-1},f^{\ast}_{N-1}, \ldots,
p_2,\emptyset,p_2^c,\emptyset,p_2,\emptyset,p_2^c,\emptyset,p_2,\ldots,
f^{\ast}_{N-1},p_{N-1},f^{\ast}_N,p); \]
proceed to define a finite sequence
$\bt_0 \approx \bt_1 \approx \cdots \bt_{N-1} \approx \bt_N = \bt_{\nvert}$.
\end{proof}

\begin{lemma}
\label{lemma:flipcork}
Consider a balanced quadriculated disk $\cD$
and a plug $p \in \cP$.
If $N$ is even and $N \ge |p|$ then
there exists an even tiling $\bt \in \cT(\cR_{-N,N;p,p})$
such that $\plug_0(\bt) = \emptyplug$;
in particular, $\bt \approx \bt_{\nvert}$.
\end{lemma}

The hypothesis $N \ge |p|$ works well with the proof 
but will not be particularly important.
In many examples a weaker hypothesis would also work;
we do not attempt to improve this bound.
The construction in the proof below will be used again,
particularly in the proof of Lemma \ref{lemma:thicksublemma}.

\begin{proof}
Consider a spanning tree for $\cD$;
this spanning tree will be kept fixed during the construction.
Define the {\em distance} between two squares of $\cD$
as measured along the spanning tree;
in particular, the distance is an even integer
if the two squares have the same color
and an odd integer otherwise.
In this proof we denote a horizontal domino
in $\cR_{-N,N;p,p}$
by $(s_a,s_b,c)$ where $s_a$ and $s_b$ are adjacent squares in $\cD$ and
$c \in \{-N+1, \ldots, N\} \subset \ZZ$ denotes the floor;
similarly, a vertical domino is denoted by $(s,c,c+1)$.

The proof is by induction on the even integer $|p|$.
The case $|p| = 0$ (so that $p = \emptyplug$) is trivial
(but perhaps too degenerate if we take $N = 0$).
In general, given $p \in \cP$ with $|p| \ge 2$,
let $\ell$ be the (odd) minimal distance
between a black square and a white square in $p$;
let $s_0$ and $s_\ell$ be squares which realize this minimum  value.
Let $\tilde p = p \smallsetminus (s_0 \cup s_\ell) \in \cP$
(with the usual abuse of notation)
so that $|\tilde p| = |p| - 2$.
Set $\tilde N = N-2$
and consider a tiling $\bt_1 \in \cT(\cR_{-N,N;p,p})$
such that $\bt_{\nvert} \approx \bt_1$,
$\plug_{-\tilde N}(\bt_1) = \plug_{\tilde N}(\bt_1) = \tilde p$
and all floors between $-\tilde N+1$ and $\tilde N$ are vertical.
The restriction of $\bt_1$
to $\cR_{-\tilde N,\tilde N;\tilde p,\tilde p}$ is
$\bt_{\nvert} \in \cT(\cR_{-\tilde N,\tilde N;\tilde p,\tilde p})$.
The desired result then follows by the induction hypothesis.
We are left with contructing $\bt_1$.

Let $s_0, s_1, \ldots, s_\ell$ be the sequence of squares,
read along the spanning tree,
from the black square $s_0$ to the white square $s_\ell$.
Notice that the minimality of $\ell$ implies that the squares
$s_1, \ldots, s_{\ell - 1}$ do not belong to the plug $p$.
The horizontal dominoes of $\bt_1$ are:
\begin{gather*}
(s_1,s_2,-N+1), (s_3,s_4,-N+1), \ldots, (s_{\ell-2},s_{\ell-1},-N+1); \\
(s_0,s_1,-N+2), (s_2,s_3,-N+2), \ldots, (s_{\ell-1},s_{\ell},-N+2); \\
(s_0,s_1,N-1), (s_2,s_3,N-1), \ldots, (s_{\ell-1},s_{\ell},N-1); \\
(s_1,s_2,N), (s_3,s_4,N), \ldots, (s_{\ell-2},s_{\ell-1},N). 
\end{gather*}
All other dominoes of $\bt_1$ are vertical,
completing the construction of $\bt_1$.
It follows from Lemma \ref{lemma:eventiling}
that $\bt \approx \bt_{\nvert}$.
\end{proof}

\begin{example}
\label{example:R2}
Figure \ref{fig:R2} illustrates this construction for $\cD = [0,4]^2$.
The even tiling $\bt_1$, as in the proof above, is constructed
(a few vertical floors are omitted).
See also Figure \ref{fig:penultimate} in Section \ref{sect:thick}.
\end{example}

\begin{figure}[ht]
\begin{center}
\includegraphics[scale=0.275]{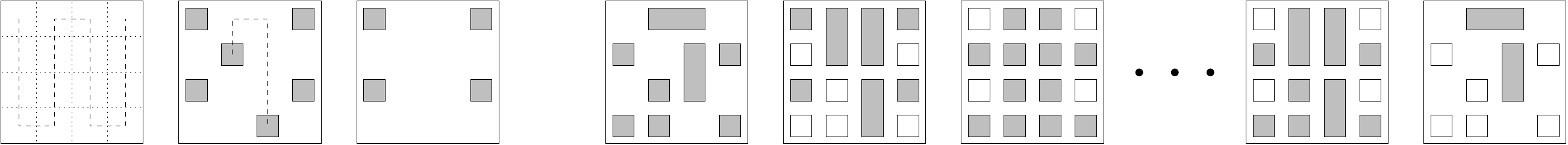}
\end{center}
\caption{A quadriculated disk $\cD$ with a spanning tree,
a plug $p$ and two squares $s_0$ and $s_5$ of opposite colors
at minimal distance $\ell = 5$ along the tree,
the plug $\tilde p$ and the tiling $\bt_1$ of $\cR_{-N,N;p,p}$.}
\label{fig:R2}
\end{figure}




\begin{lemma}
\label{lemma:cork}
Given a quadriculated region $\cD \subset \RR^2$ as above;
if $N \ge 2|\cD|$ and $p \in \cP$ then both corks
$\cR_{0,N;\emptyplug,p}$, $\cR_{0,N;p,\emptyplug}$ admit tilings.
\end{lemma}

\begin{proof}
From Lemma \ref{lemma:flipcork} there exist tilings
of $\cR_{0,N;\emptyplug,p}$ for all even $N$, $N \ge |\cD|$:
just restrict $\bt$ in the statement of  Lemma \ref{lemma:flipcork} 
to  $\cR_{0,N;\emptyplug,p}$.
In particular, taking $p = \fullplug$, there exists a tiling of
$\cR_{0,|\cD|;\emptyplug,\fullplug} = \cR_{0,|\cD|-1;\emptyplug,\emptyplug}$.
Also, there exists a vertical tiling of $\cR_{0,2;\emptyplug,\emptyplug}$;
by concatenation, there exists a tiling of $\cR_{0,N;\emptyplug,\emptyplug}$
for every $N \ge |\cD|$ (either even or odd).
Again by concatenation, there exists a tiling of
$\cR_{0,N;\emptyplug,p}$ for every $p \in \cP$ and 
every $N \ge 2|\cD|$ (either even or odd).
\end{proof}

\begin{remark}
\label{remark:cork}
Lemma \ref{lemma:cork} gives us an explicit estimate
for the minimum $N_p$ such that $N \ge N_p$
implies that $\cR_{0,N;p,\emptyplug}$ admit tilings.
We are not trying, however, to obtain the best estimate.
A few examples should convince the reader that
$N_p$ can often be taken to be significantly smaller that $2|\cD|$.
\end{remark}



\section{The domino group $G_{\cD}$ and the complex $\cC_{\cD}$}
\label{sect:groupcomplex}

The domino group has already been defined in a combinatorial way
in the Introduction, Equation \eqref{eq:dominogroup}.
Let us recall the definition.


\begin{definition}
\label{definition:dominogroup}
Consider a fixed quadriculated disk $\cD$.
Let $\cT(\cR_{\ast})$ be the set of all tilings
of all regions $\cR_{N} = \cD \times [0,N]$, $N \in \NN^\ast$.
The equivalence relation $\sim$ is defined on  $\cT(\cR_{\ast})$:
let $G_{\cD}$ be the quotient space $\cT(\cR_{\ast})/\sim$.
The set $G_{\cD}$ has a group structure.
Indeed, the identity is the vertical tiling $\bt_{\nvert}$ of $\cR_2$.
The product of two tilings
$\bt_1 \in \cT(\cR_{N_1})$ and $\bt_2 \in \cT(\cR_{N_2})$
is the concatenation $\bt_1\ast \bt_2 \in \cT(\cR_{N_1+N_2})$.
The inverse of $\bt_1 \in \cT(\cR_{N_1})$ is obtained by reflection
on the horizontal plane.
\end{definition}

In this section we construct another of our main objects:
the complex $\cC_{\cD}$.
As we shall see in Equation \eqref{equation:dominogroup},
$G_{\cD}$ equals the fundamental group of $\cC_{\cD}$,
giving us an equivalent definition of the domino group
in another language.

Given a quadriculated disk $\cD$, we construct a $2$-complex $\cC_{\cD}$.
The vertices of $\cC_{\cD}$ are the plugs $p \in \cP$.
We first construct a graph, which is almost (but not quite)
the same thing as constructing the $1$-skeleton of $\cC_{\cD}$.
The (undirected) edges of $\cC_{\cD}$ between two vertices (plugs)
$p_0, p_1 \in \cP$ are valid (full) floors of the form
$f = (p_0, f^{\ast}, p_1)$.
Thus, if $p_0$ and $p_1$ are not disjoint there is no edge between them;
if $p_0$ and $p_1$ are disjoint, there is one edge for each tiling
of $\cD_{p_0,p_1}$.
Notice that if $p_0 \ne p_1$ then
$f = (p_0, f^{\ast}, p_1)$ and $f^{-1} = (p_1, f^\ast, p_0)$
define two orientations of the same edge.
Notice also that there is one loop
from $\emptyplug$ to itself
for each tiling of $\cD$;
there are no other loops in $\cC_{\cD}$.

There is a natural identification between tilings
of the cork $\cR_{N;p_0,p_N}$ 
and paths of length $N$ in $\cC_{\cD}$ from $p_0$ to $p_N$.
Tilings of $\cR_N$ correspond to
closed paths of length $N$ in $\cC_{\cD}$ from $\emptyplug$ to itself.
Indeed, both our usual figures of tilings
(for instance, Figure \ref{fig:noflip8})
and descriptions as lists of plugs and floors
(as in Equation \ref{equation:floorsandplugs})
can be directly interpreted as paths:
each floor is an edge,
the initial vertex of each edge is a plug indicated by white squares
and the final vertex is a plug indicated by black squares.

Loops from $\emptyplug$ to itself
are interpreted in the graph theoretical sense:
if we go from $\emptyplug$ to $\emptyplug$ in one move,
we must specify which loop (i.e., which floor) is used, and nothing else.
In particular, we do not have to specify what ``orientation''
of the loop was used;
consistently, no such ``orientation'' exists for a horizontal floor.
In this sense a graph with loops is not quite the same as a $1$-complex:
this difficulty will be addressed by attaching certain $2$-cells (see below).
Let us stay with the graph theoretical point of view for a little longer.

It follows from Lemma \ref{lemma:cork} that $\cC_{\cD}$ is path connected
and not bipartite:
for sufficiently large $N$ there exist paths of length $N$
from any initial vertex to any final vertex.
Let $A = A_{\cD} \in \ZZ^{\cP \times \cP}$
be the adjacency matrix of $\cC_{\cD}$.
We have $A_{(p,\tilde p)} = |\cT(\cD_{p,\tilde p})|$;
we set $|\cT(\cD_{p,\tilde p})| = 0$
if $p$ and $\tilde p$ are not disjoint.
From the description of tilings of corks as paths in $\cC_{\cD}$ we have
\begin{equation}
\label{eq:count}
|\cT(\cR_{0,N;p_{0},p_{N}})| =
(A^{N})_{(p_0,p_N)} =  
\sum_{(p_1,\ldots,p_{N-1}) \in \cP^{N-1}}
\left( \prod_{1 \le j \le N} |\cT(\cD_{p_{j-1},p_j})| \right).
\end{equation}


\goodbreak

\begin{lemma}
\label{coro:A1Npos}
Consider a balanced quadriculated disk $\cD \subset \RR^2$;
let $A = A_{\cD}$ be the adjacency matrix of $\cC_{\cD}$.
If $N \ge 4|\cD|$ then all entries of $A^N$ are strictly positive.
\end{lemma}

\begin{proof}
This is equivalent to saying that for all $N \ge 4|\cD|$
and for all  $p, \tilde p \in \cP$
the cork $\cR_{N;p,\tilde p}$ admits a tiling.
Consider plugs $p, \tilde p \in \cP$.
Use Lemma \ref{lemma:cork} to obtain
$\bt \in \cT(\cR_{0,2|\cD|;p,\emptyplug})$ and
$\tilde\bt \in \cT(\cR_{2|\cD|,N;\emptyplug,\tilde p})$
(the second one requires a translation by $(0,0,2|\cD|)$).
Concatenate $\bt$ and $\tilde\bt$ to obtain the desired tiling
of $\cR_{0,N;p,\tilde p}$.
\end{proof}

We complete the construction of the complex $\cC_{\cD}$
by attaching $2$-cells.
First, we adress the fact that a graph with loops
is not quite the same thing as a $1$-complex.
We solve this by attaching a disk with boundary $f\ast f$
for each floor $f$ of the form $f = (\emptyplug,f^\ast,\emptyplug)$,
i.e., for each loop:
this guarantees that $f$ and $f^{-1}$ are now homotopic.
The other $2$-cells correspond to flips,
and it is convenient to describe separately
horizontal flips, which will correspond to bigons,
and vertical flips, which will correspond to quadrilaterals.
Let $p_0, p_1 \in \cP$ be disjoint plugs;
let $f^{\ast}_a$ and $f^{\ast}_b$ be tilings of $\cD_{p_0, p_1}$
joined by a flip.
Attach
a bigon with vertices $p_0$ and $p_1$ and edges
$f_a = (p_0,f^{\ast}_a,p_1)$ and $f_b = (p_0,f^{\ast}_b,p_1)$;
see example in Figure \ref{fig:hflip}.

\begin{figure}[ht]
\centering
\def\svgwidth{120mm}
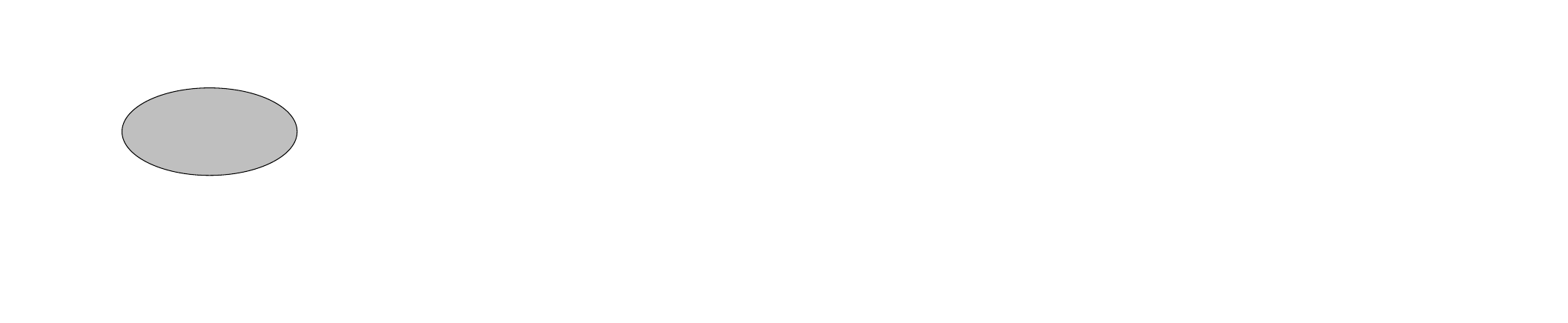
\caption{A horizontal flip defines a $2$-cell in $\cC_{\cD}$.}
\label{fig:hflip}
\end{figure}

Let $p_0, p_1, \tilde p_1, p_2 \in \cP$ be plugs.
Assume that $p_1$ is obtained from $\tilde p_1$ by removing
two adjacent squares;
let $d$ be the domino formed by these two squares.
Assume that $\tilde p_1$ is
disjoint from both $p_0$ and $p_2$;
notice that this implies that $p_1$ is likewise 
disjoint from both $p_0$ and $p_2$.
Let $\tilde f^{\ast}_1$ and $\tilde f^{\ast}_2$
be tilings of $\cD_{p_0,\tilde p_1}$ and $\cD_{\tilde p_1,p_2}$,
respectively.
Let $f^{\ast}_1$ and $f^{\ast}_2$
be tilings of $\cD_{p_0, p_1}$ and $\cD_{p_1,p_2}$
obtained from $\tilde f^{\ast}_1$ and $\tilde f^{\ast}_2$,
respectively,
by adding the domino $d$.
Attach
a quadrilateral with vertices
$p_0, p_1, \tilde p_1, p_2$ and edges
$f_1 = (p_0,f^{\ast}_1,p_1)$,
$\tilde f_1 = (p_0,\tilde f^{\ast}_1,\tilde p_1)$,
$f_2 = (p_1,f^{\ast}_2,p_2)$ and
$\tilde f_2 = (\tilde p_1, \tilde f^{\ast}_2, p_2)$;
see an example in Figure \ref{fig:vflip}.

\begin{figure}[ht]
\centering
\def\svgwidth{120mm}
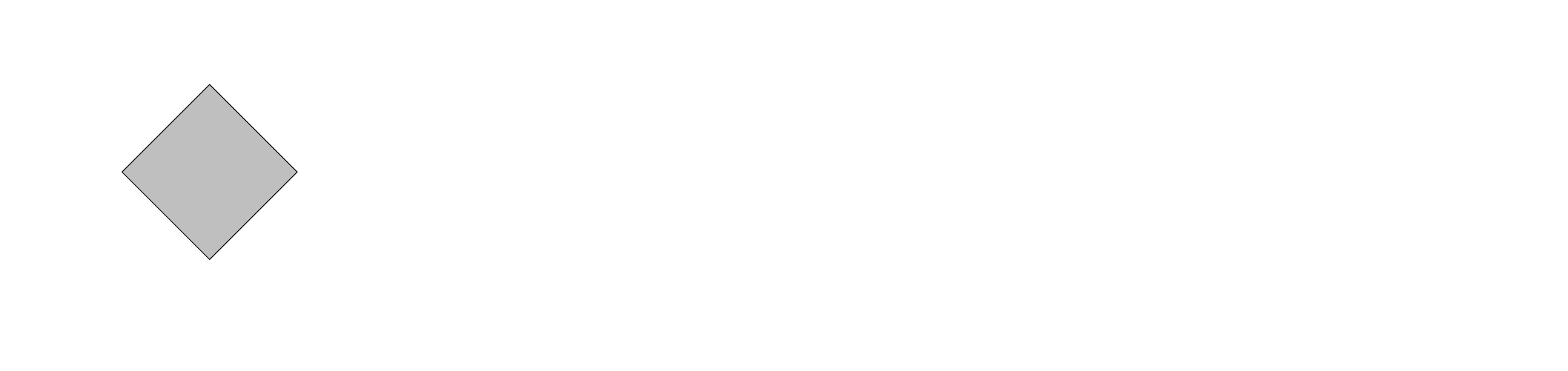
\caption{A vertical flip defines a $2$-cell in $\cC_{\cD}$.}
\label{fig:vflip}
\end{figure}

The above identification between tilings of corks
and paths in $\cC_{\cD}$ gives us a natural concatenation operation.
Let $p_0, p_1, p_2 \in \cP$ be plugs.
Let $\cR_{01} = \cR_{N_0,N_1;p_0,p_1}$, $\cR_{12} = \cR_{N_1,N_2;p_1,p_2}$
and $\cR_{02} = \cR_{N_0,N_2;p_0,p_2}$ be corks.
If $\bt_{01} \in \cT(\cR_{01})$ and $\bt_{12} \in \cT(\cR_{12})$
are tilings,
concatenate them to define
$\bt_{02} = \bt_{01} \ast \bt_{12} \in \cT(\cR_{02})$.
The dominoes of $\bt_{02}$ are:
the dominoes of $\bt_{01}$,
the dominoes of $\bt_{12}$,
vertical dominoes of the form $s \times [N_1-1,N_1+1]$
for $s \subset \cD$ a square in the plug $p_1$.
We thus have
$\ast: \cT(\cR_{01}) \times \cT(\cR_{12}) \to \cT(\cR_{02})$,
\[ \cT(\cR_{01}) \ast \cT(\cR_{12}) =
\{ \bt \in \cT(\cR_{02}) \;|\; \plug_{N_1}(\bt) = p_1 \}. \]

\begin{lemma}
\label{lemma:movevert}
Let $\cD$ be a balanced quadriculated disk;
let $p_0, p_1 \in \cP$ be plugs;
let $\bt \in \cT(\cR_{N_0,N_1;p_0,p_1})$ be a tiling.
Let
$\bt_{\nvert,p_0} \in \cT(\cR_{N_0-2,N_0;p_0,p_0})$
and 
$\bt_{\nvert,p_1} \in \cT(\cR_{N_1,N_1+2;p_1,p_1})$
be the vertical tilings in the corks above.

Let $\tilde\bt_0  = \bt_{\nvert,0} \ast \bt \in \cT(\cR_{N_0-2,N_1;p_0,p_1})$,
 $\tilde\bt_1 = \bt \ast \bt_{\nvert,1} \in \cT(\cR_{N_0,N_1+2;p_0,p_1})$.
Apply a translation by $(0,0,2)$ (and abuse notation):
$\tilde\bt_0 \in \cT(\cR_{N_0,N_1+2;p_0,p_1})$.
Then $\tilde\bt_0 \approx \tilde\bt_1$.
\end{lemma}

\begin{proof}
In order to transform $\tilde\bt_0$ into $\tilde\bt_1$
it is better to focus on the horizontal dominoes
and think of the vertical dominoes as background.
We have to move every horizontal domino down by two floors.
This is performed in increasing order of the $z$ coordinate,
thus guaranteeing that at the time a horizontal domino must move
the four unit cubes one or two floors below it
are filled in by two vertical dominoes.
Two flips then have the desired effect of moving down
the horizontal domino.
\end{proof}

In the introduction we gave a combinatorial definition of
the equivalence relation $\sim$ which is weaker than $\approx$;
we now extend it to tilings of corks.
Consider $\bt_0 \in \cT(\cR_{0,N_0;p_a,p_b})$ and
$\bt_1 \in \cT(\cR_{0,N_1;p_a,p_b})$ with $N_0$ and $N_1$ of the same parity.
We write $\bt_0 \sim \bt_1$ if and only if
there exist $M_0, M_1 \in 2\NN$ with $N_0 + M_0 = N_1 + M_1$ and
$\bt_0 \ast \bt_{\nvert,0} \approx \bt_1 \ast \bt_{\nvert,1}$
where
$\bt_{\nvert,i}$ is the vertical tiling of $\cR_{N_i,N_i+M_i;p_b,p_b}$.
The following lemma gives an alternative definition
which is topological and natural.



\begin{lemma}
\label{lemma:homotopy}
Consider $\bt_0 \in \cT(\cR_{0,N_0;p_a,p_b})$ and
$\bt_1 \in \cT(\cR_{0,N_1;p_a,p_b})$:
we have $\bt_0 \sim \bt_1$ if and only if the paths
$\bt_0$ and $\bt_1$ from $p_a$ to $p_b$ are homotopic
with fixed endpoints.
\end{lemma}

\begin{proof}
First notice that the moves in the definition of $\sim$
(applying flips, adding or deleting pairs of vertical floors at the endpoint)
are examples of homotopies with fixed endpoints,
proving one implication.

For the other direction, notice that
a homotopy with fixed endpoints between $\bt_0$ and $\bt_1$
allows for flips since in the definition of $\cC_{\cD}$
there are $2$-cells corresponding to flips.
A homotopy allows for an extra operation.
At any vertex (plug) $p_i$ and for any edge (floor)
$f = (p_i,f^\ast,\tilde p)$ we may add two new edges to the path,
thus modifying the path from 
$\bt = (\ldots,f_i,p_i,f_{i+1},\ldots)$ (of length $N$) to
$\tilde\bt = (\ldots,f_i,p_i,f,\tilde p,f^{-1},p_i,f_{i+1},\ldots)$
(of length $N+2$).
We may also do the same operation in reverse.
We must therefore check that $\bt \sim \tilde\bt$.

Indeed, add two vertical floors at the end of $\bt$.
Use Lemma \ref{lemma:movevert} to move (by flips) the two vertical floors
to position $i$, thus obtaining the path (tiling)
$(\ldots,f_i,p_i,f_{\nvert},p_i^c,f_{\nvert},p_i,f_{i+1},\ldots)$
(of length $N+2$).
The above path is equivalent to $\tilde\bt$ by vertical flips,
as discussed in the proof of Lemma \ref{lemma:eventiling}.


Also, the $2$-cells allow for a few other operations which may,
at first, look new and requiring checking.
For instance, the $2$-cell in Figure \ref{fig:vflip}
was introduced to allow us to move from
$(\ldots,p_0,f_1,p_1,f_2,p_2,\ldots)$ to
$(\ldots,p_0,\tilde f_1,\tilde p_1,\tilde f_2,p_2,\ldots)$
(or vice-versa), which is a vertical flip
and therefore consistent with $\sim$ (and even with $\approx$).
This cell also has the perhaps unexpected effect of allowing us to move from
$(\ldots,p_1,f_1^{-1},p_0,\tilde f_1,\tilde p_1, \ldots)$ to
$(\ldots,p_1,f_2,p_2,\tilde f_2^{-1},\tilde p_1, \ldots)$.
It is not hard to verify that this example is also consistent with $\approx$.
A case by case study reveals that all moves are indeed consistent with $\sim$;
alternatively, adding and deleting pairs of matching floors
shows that the previous paragraphs already provide a complete proof.
\end{proof}

Given a balanced quadriculated disk $\cD$,
we defined in the introduction the {\em domino group $G_{\cD}$}:
elements of the group are tilings modulo $\sim$ and
the operation is concatenation.
It follows from Lemma \ref{lemma:homotopy} and
Definition \ref{definition:dominogroup} that
\begin{equation}
\label{equation:dominogroup}
G_{\cD} = \pi_1(\cC_{\cD},\emptyplug),
\end{equation}
the fundamental group of the complex $\cC_{\cD}$
with base point $\emptyplug \in \cP$.

There is a homomorphism $G_{\cD} \to \ZZ/(2)$
taking $\bt \in \cT(\cR_{N})$ to $N \bmod 2$.
The normal subgroup $G^{+}_{\cD} < G_{\cD}$
is the kernel of this homomorphism.
If $\cD$ is tileable, a tiling $\bt_c$ of $\cR_1$
defines an element $c \in G_{\cD}$ of order $2$:
indeed, $\bt_c \ast \bt_c \approx \bt_{\nvert} \in \cT(\cR_2)$.
Let $H = \{e, c\} < G_{\cD}$:
we have that $G_{\cD}$ is then a semidirect product
$G_{\cD} = G^{+}_{\cD} \ltimes H$;
as we shall see, this is sometimes but not always a direct product.

Consider the corresponding double cover $\cC^{+}_{\cD}$.
The set of vertices of $\cC^{+}_{\cD}$ is 
$\cP^{+} = \cP \times \ZZ/(2)$:
its elements are {\em plugs with parity}:
a plug $p$ with the extra information of the parity of its position.
Similarly, an oriented edge of $\cC^{+}_{\cD}$ is a {\em floor with parity},
a pair $(f,k)$ where $k \in \ZZ/(2)$ indicates the parity of its position;
notice that $(f,k)^{-1} = (f^{-1},k+1)$.

It follows from Lemma \ref{lemma:trivialdisk}
that if $\cD$ is trivial than
the domino group $G_{\cD}$ is isomorphic to $\ZZ/(2)$.
As we shall see in the next section,
if $\cD$ is nontrivial then $G_{\cD}$ is infinite.


\section{Twist}
\label{sect:twist}

Following \cite{segundoartigo},
we first recall the of the twist
of a tiling $\bt \in \cT(\cR_N)$
(there are other definitions for other regions,
as discussed in \cite{FKMS}).
For a domino $d$,
let $v(d)  \in \{\pm e_1, \pm e_2, \pm e_3\} \subset \RR^3$
be the unit vector
from the center of the white cube to the center of the black cube of $d$.
For $u \in \{\pm e_1, \pm e_2\}$,
define the \emph{$u$-shade} of
$X \subset \RR^3$ to be
\[ \cS^{u}(X) = \interior((X + [0,+\infty) u) \smallsetminus X);
\quad
X + [0,+\infty) u = \{x+tu; x \in X, t \in [0,+\infty)\}. \]
(The case $u = \pm e_3$,
which is also discussed in \cite{segundoartigo},
will not be considered in the present paper.)
Given a tiling $\bt \in \cT(\cR)$ and
two dominoes $d_0$ and $d_1$ of $t$,
define the {\em effect of $d_0$ on $d_1$ along $u$} as
$\tau^{u}(d_0,d_1) \in \{0,\pm\frac14\}$:
\[ \tau^{u}(d_0,d_1) = 
\begin{cases}
\frac14 \det(v(d_1),v(d_0),u), & d_1 \cap \cS^{u}(d_0) \ne \emptyset, \\
0, & \textrm{otherwise.}
\end{cases} \]
The {\em twist} of $\bt$ is
\[ \Tw(\bt) = \sum_{d_0,d_1 \in \bt} \tau^{u}(d_0,d_1). \]
It is shown in \cite{segundoartigo} that the value of $\Tw(\bt)$
is always an integer and that it does not depend
on the choice of $u \in \{\pm e_1, \pm e_2\}$
(see also Remark \ref{remark:cocycle}).

\begin{figure}[ht]
\centering
\def\svgwidth{114mm}
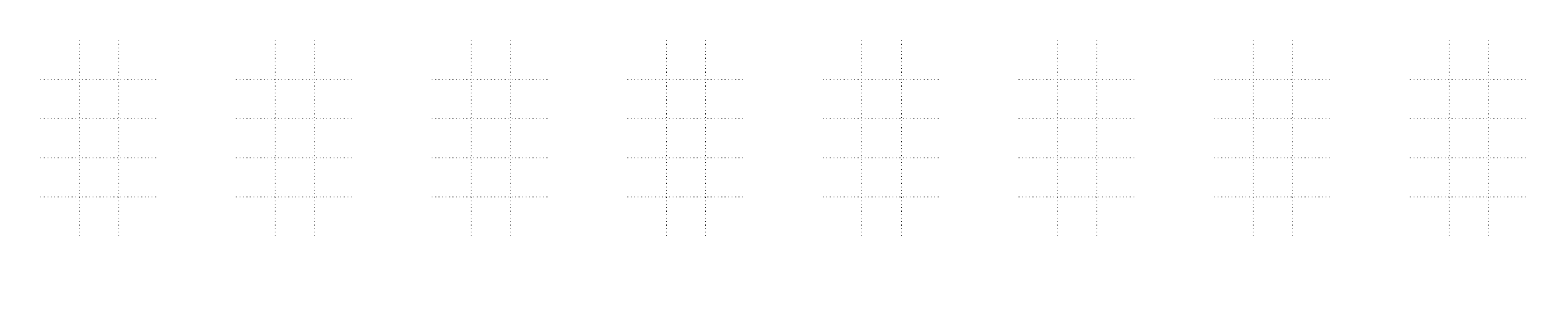
\caption{The value of $4\tau^{e_2}(d,s)$ in eight examples.
The sign depends on three bits:
horizontal position of domino, parity of distance
and relative position.}
\label{fig:tau}
\end{figure}

Notice that $\tau^{u}(d_0,d_1) = 0$ unless,
for some $i \in \{0,1\}$ and for some $j$,
we have that
$d_i$ is horizontal and contained in $\cD \times [j-1,j]$
and $d_{1-i}$ is vertical and intersects $\cD \times [j-1,j]$.
Given a planar domino $d \subset \cD$ and a unit square $s \subset \cD$
with disjoint interiors, define
\[ \tau^u(d,s) =
\tau^u(d \times [0,1], s \times [0,2]) +
\tau^u(s \times [0,2], d \times [0,1]) \in \{0,\pm\frac14\}; \]
examples are given in Figure \ref{fig:tau}.
Given two disjoint plugs $p, \tilde p \in \cD$ and
a planar tiling $f \in \cT(\cD_{p,\tilde p})$
define
\begin{equation}
\label{equation:cocycle}
\tau^u(f,p) = \sum_{d \in f, s \in p} \tau^u(d,s)
\in \frac14\ZZ; \qquad
\tau^u(f; p, \tilde p) = \tau^u(f,\tilde p) - \tau^u(f,p) \in \frac14 \ZZ.
\end{equation}
Notice that $\tau^u(f; \tilde p, p) = -\tau^u(f; p, \tilde p)$.
For $\bt \in \cT(\cR_N)$ we clearly have
\[ \Tw(\bt) = \sum_{0 < j \le N}
\tau^u(\floorop_j(\bt); \plug_{j-1}(\bt), \plug_j(\bt)). \] 
The values of $\tau^u(f,p)$ and $\tau^u(f;p,\tilde p)$
usually depend on the choice of $u$
and neither is necessarily an integer.
Other papers
(including \cite{FKMS}, \cite{primeiroartigo} and \cite{segundoartigo})
discuss several ways to think about the twist;
we shall see another way in this paper.

The twist defines a homomorphism $\Tw: G_{\cD} \to \ZZ$.
Recall that if $\cD$ is trivial then $G_{\cD} \approx \ZZ/(2)$
and therefore the map $\Tw$ is constant equal to $0$.

\begin{remark}
\label{remark:cocycle}
Equation \ref{equation:cocycle} defines $\tau^u$
as a function taking oriented edges of $\cC_{\cD}$ to real numbers. 
In the language of homology, $\tau^u \in C^1(\cC_{\cD};\RR)$.
It is not hard to check that $\tau^u \in Z^1 \subset C^1$,
i.e., that $\tau^u$ is a cocycle;
also, $\tau^{e_1} - \tau^{e_2} \in B^1 \subset Z^1$.
The two cocycles $\tau_{e_1}$ and $\tau_{e_2}$
therefore define the same element of the cohomology:
$[\tau^{e_1}] = [\tau^{e_2}] \in H^1(\cC_{\cD};\RR)$.
It is well known that thare exists a natural isomorphism
$\Hom(\pi_1(X);\ZZ) \approx H^1(X)$.
By following the construction of this isomorphism, we see that
$\Tw \in \Hom(\pi_1(\cC_{\cD});\ZZ)$ is taken to $[\tau^u]$
(for either $u = e_1$ or $u = e_2$).
This provides us with an alternative justification
for the facts that $\Tw$ does not depend on the choice of $u$
and is invariant under flips.
\end{remark}


\begin{lemma}
\label{lemma:nontrivialtwist}
Let $\cD$ be a nontrivial balanced quadriculated disk
and $N \ge 4|\cD| + 3$.
There exist tilings $\bt_0$ and $\bt_1$ of $\cR_N = \cD \times [0,N]$
with $\Tw(\bt_1) = \Tw(\bt_0) + 1$.
In particular, the restriction $\Tw: G^{+}_{\cD} \to \ZZ$
is surjective.
\end{lemma}

Again, we are not trying to obtain sharp estimates.
For $\cD = [0,2] \times [0,3]$,
there exist tilings of twists $-1$, $0$ and $1$
in $\cR_N = [0,2] \times [0,3] \times [0,N]$ for $N \ge 3$,
as illustrated in Figure \ref{fig:233} for $N = 3$.
As another example,
Figure \ref{fig:girafa} shows tilings of twists $-1$, $0$ and $1$
of $\cR_5$
for another nontrivial quadriculated disk $\cD$.

\begin{figure}[ht]
\begin{center}
\includegraphics[scale=0.275]{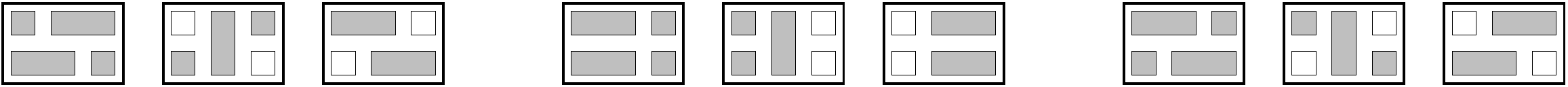}
\end{center}
\caption{Tilings of twist $-1$, $0$ and $1$ of the box
$[0,2] \times [0,3] \times [0,3]$.}
\label{fig:233}
\end{figure}

\begin{figure}[ht]
\begin{center}
\includegraphics[scale=0.275]{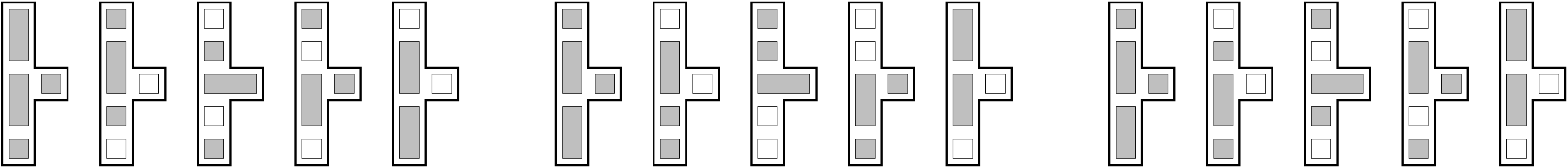}
\end{center}
\caption{Tilings of twist $-1$, $0$ and $1$ of the region
$\cD \times [0,5]$ for a nontrivial quadriculated disk $\cD$.}
\label{fig:girafa}
\end{figure}

\begin{proof}
Assume without loss of generality that at least one square has neighbors
in the directions $e_1$ and $\pm e_2$.
Figure \ref{fig:nontrivialtwist} shows floors
$2|\cD|+1$, $2|\cD|+2$ and $2|\cD|+3$ for tilings $\bt_0$ and $\bt_1$.
The other squares are filled in by vertical dominoes
filling floors $2|\cD|+1$ and $2|\cD|+2$.
Lemma \ref{lemma:cork} guarantees that the first $2|\cD|$ floors
and the last $N - (2|\cD| + 3) \ge 2|\cD|$ can be consistently tiled:
tile them in the same way for $\bt_0$ and $\bt_1$.
Set
\[ d_j = 
\tau^{e_2}(\floorop_j(\bt_1); \plug_{j-1}(\bt_1), \plug_j(\bt_1)) -
\tau^{e_2}(\floorop_j(\bt_0); \plug_{j-1}(\bt_0), \plug_j(\bt_0)). \]
We have $d_j = 1$ for $j = 2|\cD|+2$ and $d_j = 0$ otherwise.
Thus $\Tw(\bt_1) - \Tw(\bt_0) = 1$, completing the proof.
\end{proof}

\begin{figure}[ht]
\begin{center}
\includegraphics[scale=0.275]{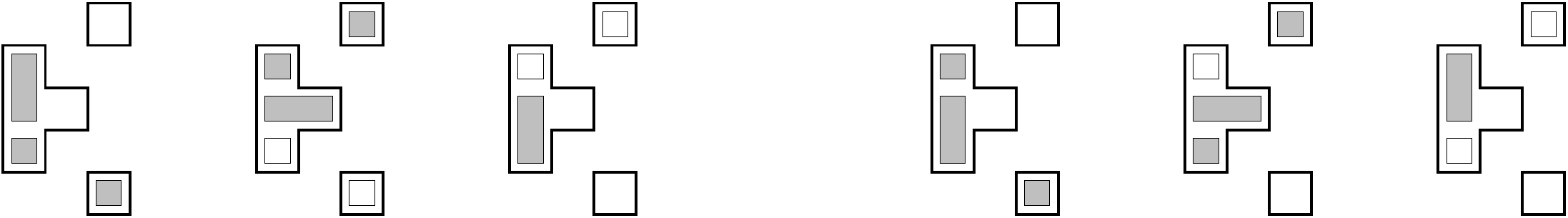}
\end{center}
\caption{Two tiling with twist differing by one.
The central square in the first column has at least three neighbors.
The two detached squares must exist (in some position in $\cD$)
since $\cD$ is balanced.
Dominoes which are not shown are the same for both tilings.}
\label{fig:nontrivialtwist}
\end{figure}

Recall from the Introduction that
a quadriculated disk $\cD$ is {\em regular} if the map
$\Tw: G^{+}_{\cD} \to \ZZ$ is an isomorphism:
it then follows that
$G_{\cD}$ is isomorphic to $\ZZ \oplus (\ZZ/(2))$.
The aim of Sections \ref{sect:thin} to \ref{sect:thick}
is to prove Theorem \ref{theo:rectangle}:
a rectangle $[0,L] \times [0,M]$ is regular
if and only if $\min\{L,M\} \ge 3$.



\section{Thin rectangles}
\label{sect:thin}

Consider $\cD = [0,2] \times [0,3]$
and the two tilings $\bt_0, \bt_1 \in \cT(\cR_4)$
shown in Figure \ref{fig:234}.
The tilings satisfy
$\Tw(\bt_0) = \Tw(\bt_1) = +1$ and $\bt_0 \not\sim \bt_1$.
The second claim is not obvious;
it follows from the main result in this section.

\begin{figure}[ht]
\begin{center}
\includegraphics[scale=0.275]{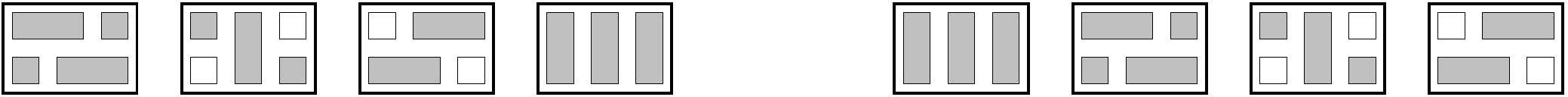}
\end{center}
\caption{Two tilings $\bt_0$ and $\bt_1$ of $\cR_4$
for $\cD = [0,2] \times [0,3]$.
We have $\Tw(\bt_0) = \Tw(\bt_1) = +1$ and $\bt_0 \not\sim \bt_1$.}
\label{fig:234}
\end{figure}

Let $F_2$ be the free group in $2$ generators $a$ and $b$.
Let $\psi: \ZZ/(2) \to \Aut(F_2)$ be defined by
$\psi(1)(a) = b^{-1}$, $\psi(1)(b) = a^{-1}$;
Let $G_2 = F_2 \ltimes_{\psi} \ZZ/(2)$ be the corresponding semidirect product;
in other words, a presentation of $G_2$ is:
\begin{equation}
\label{eq:presentation}
G_2 = \langle a, b, c | c^2 = e, cac = b^{-1}, cbc = a^{-1} \rangle.
\end{equation}

\begin{lemma}
\label{lemma:thin}
Let $\cD = [0,2] \times [0,M]$, $M \ge 3$.
Then there exists a surjective homomorphism $\phi: G_{\cD} \to G_2$.
\end{lemma}


As we shall see, $\cD = [0,2]\times [0,3]$
and $\bt_0, \bt_1 \in \cT(\cR_4)$ as in Figure \ref{fig:234},
we have $\phi(\bt_0) = a$ and $\phi(\bt_1) = b^{-1}$,
which implies $\bt_0 \not\sim \bt_1$.

The idea of the proof is to construct an explicit map
taking oriented edges of $\cC_{\cD}$ to $G_2$.
Most edges are taken to the identity;
the exceptions are explicitely listed in
Figures \ref{fig:oddfree} and \ref{fig:evenfree}
and in Table \ref{tab:phree}.
These edges belong to the boundary of a small number of $2$-cells
and it is therefore easy to check that the boundary of such cells
is taken to the identity.

\begin{proof}
We first construct the restriction $\phi: G^{+}_{\cD} \to F_2$
by working in $\cC_{\cD}^{+}$,
the double cover of $\cC_{\cD}$
constructed at the end of Section \ref{sect:groupcomplex}.
We provide an explicit map taking each floor with parity
$\bff = (f,k) = (p_0,f^{\ast},p_1,k)$
to $\phi(\bff) \in \{e,a,a^{-1},b,b^{-1}\} \subset F_2$
(recall that floors with parity are edges of $\cC^{+}_{\cD}$).
Most floors are taken to the identity element $e$.
It is helpful to consider first the case $M$ odd, then the case $M$ even.
For odd $M$,  $\phi(\bff) = e$ {\em unless}:
the central column $[0,2] \times [\frac{M-1}{2},\frac{M+1}{2}]$
is occupied by a domino of $f^{\ast}$ {\em and}
the plug $p_0$ marks exactly $\frac{M-1}{2}$ squares in the region
$[0,2] \times [0,\frac{M-1}{2}]$, all of the same color.
It then follows that all squares in
$\cD \smallsetminus [0,2] \times [\frac{M-1}{2},\frac{M+1}{2}]$
are marked by either $p_0$ and $p_1$,
and the marking follows a checkerboard patterns.
Given odd $M$, there exist only two floors $f_0$ and $f_1 = f_0^{-1}$
(and four signed floors)
satisfying the conditions above,
shown in Figure \ref{fig:oddfree} for $M = 7$.
Set $\bff_0 = (f_0,0)$ and $\bff_1 = (f_1,0)$
so that $\bff_0^{-1} = (f_1,1)$ and $\bff_1^{-1} = (f_0,1)$.

\begin{figure}[ht]
\begin{center}
\includegraphics[scale=0.275]{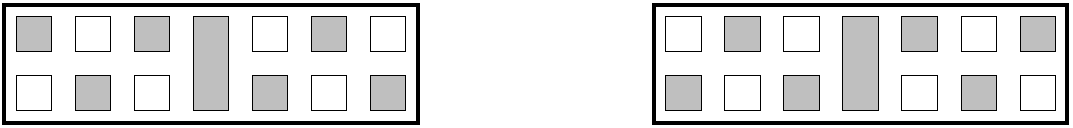}
\end{center}
\caption{The floors $f_0$ and $f_1$
for $\cD = [0,2] \times [0,7]$ for which $\phi$ is non trivial.}
\label{fig:oddfree}
\end{figure}

Finally, set $\phi(\bff_0) = a$, $\phi(\bff_0^{-1}) = a^{-1}$,
$\phi(\bff_1) = b$ and $\phi(\bff_1^{-1}) = b^{-1}$.
In order to verify that $\phi$ indeed defined a group homomorphism
$\phi: G_{\cD} = \pi_1(\cC_{\cD}) \to F_2$
we must verify that for every $2$-cell
the oriented boundary is taken to $e$.
Since neither $f_0$ nor $f_1$ are part of the boundary of any $2$-cell,
the result follows.

For even $M$,  $\phi(\bff) = e$ {\em unless}
exactly one of the following conditions hold:
\begin{enumerate}
\item{the column $[0,2] \times [\frac{M}{2}-1,\frac{M}{2}]$
is occupied by a domino of $f^{\ast}$ {\em and}
the plug $p_0$ marks exactly $\frac{M}{2}-1$ squares in the region
$[0,2] \times [0,\frac{M}{2}-1]$, all of the same color.}
\item{the column $[0,2] \times [\frac{M}{2},\frac{M}{2}+1]$
is occupied by a domino of $f^{\ast}$ {\em and}
the plug $p_0$ marks exactly $\frac{M}{2}-1$ squares in the region
$[0,2] \times [\frac{M}{2}+1,M]$, all of the same color.}
\end{enumerate}
There are then four classes of floors for which $\phi$ is non trivial,
shown in Figure \ref{fig:evenfree};
call them class $0$, $1$, $2$ and $3$ (in the order shown).

\begin{figure}[ht]
\begin{center}
\includegraphics[scale=0.275]{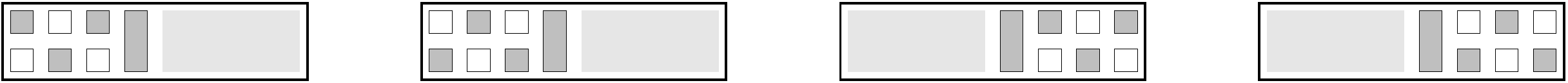}
\end{center}
\caption{Four classes of floors
for $\cD = [0,2] \times [0,8]$.}
\label{fig:evenfree}
\end{figure}

If $\bff = (f,k)$ is of class $j$, the value of $\phi(\bff)$
is shown in Table \ref{tab:phree}.

\begin{table}[ht]
\begin{center}
\begin{tabular}{c | c c c c c c c c}
$(j,k)$ & $(0,0)$ & $(1,0)$ & $(2,0)$ & $(3,0)$ &
$(0,1)$ & $(1,1)$ & $(2,1)$ & $(3,1)$ \\
[0.5ex] \hline
$\vphantom{\frac{\tilde f}{2}}\phi(\bff)$ & $a$ & $b$ & $a^{-1}$ & $b^{-1}$ &
$b^{-1}$ & $a^{-1}$ & $b$ & $a$ \\ 
    \end{tabular}
\caption{Values of $\phi(\bff)$.}
\label{tab:phree}
\end{center}
\end{table}

Notice that the floors shown in Figure \ref{fig:badfloors}
satisfy both conditions above:
by definition, $\phi(\bff) = e$
if $\bff = (f,k)$ for either example
and for either value of $k$.

\begin{figure}[ht]
\begin{center}
\includegraphics[scale=0.275]{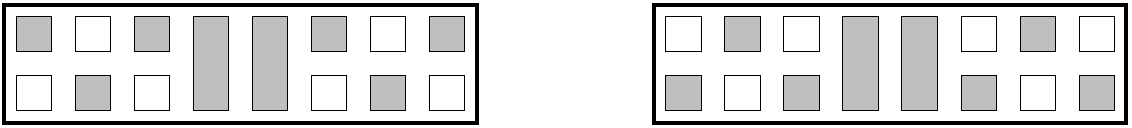}
\end{center}
\caption{Floors for $\cD = [0,2] \times [0,8]$.}
\label{fig:badfloors}
\end{figure}

A case by case check shows that $\phi$ extends to a homomorphism
since it takes the boundary of any $2$-cell to $e$.
This completes the construction of the homomorphisms
$\phi: G_{\cD} \to G_2$ for $\cD = [0,2] \times [0,M]$, $M \ge 3$.
In all cases, it is not hard to create tilings $\bt \in \cR_N$
(with $N$ even) such that $\phi(\bt) = a$ or $\phi(\bt) = b$;
for the first two tilings $\bt_0, \bt_1 \in \cT(\cR_6)$
in Figure \ref{fig:L2} we have
$\phi(\bt_0) = a^{-1}$ and
$\phi(\bt_1) = b$ (with $M = 5$).
This, by the way, proves what we claimed back then in the Introduction;
notice that $\phi(\bt_2) = e$.

Finally, we extend the map to $G_{\cD}$ by taking a tiling
$\bt \in \cT_{\cR_1}$ to the generator $c$ of $H \approx \ZZ/(2)$;
it is easy to check that the relations in Equation \ref{eq:presentation}
indeed hold.
\end{proof}

\begin{figure}[h!]
\begin{center}
\includegraphics[scale=0.275]{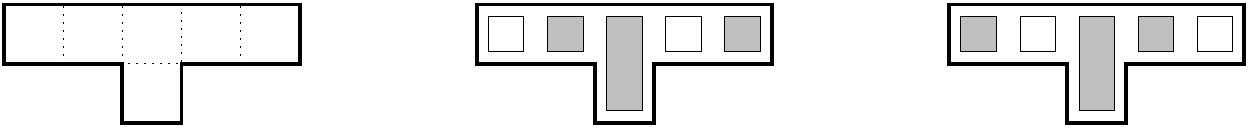}
\end{center}
\caption{Another example of a quadriculated disk which is not regular.}
\label{fig:nonregulart}
\end{figure}

For $M = 3$ and $M = 4$ the map $\phi$ constructed above is an isomorphism.
For $M \ge 5$ this is not the case;
indeed, $\Tw|_{\ker(\phi) \cap G^{+}_{\cD}}$ is surjective.
We do not provide a proof of this claim here,
but it follows easily from  \cite{primeiroartigo}.
The construction above is inspired by the results of \cite{primeiroartigo};
we focus here on the coefficients of highest degree
of the Laurent polynomial $P_{\bt} \in \ZZ[q,q^{-1}]$.

\goodbreak

\begin{remark}
\label{remark:nonregulart}
A construction similar to the proof of Lemma \ref{lemma:thin}
proves that a few other quadriculated disks are likewise not regular.
Consider, for instance, the disk $\cD$ in Figure \ref{fig:nonregulart}.
The floors $f_0$ and $f_1$ shown in the same figure
have the same property as the floors of the same name in the proof
for rectangles $[0,2]\times[0,M]$, $M$ odd.
Indeed, it is easy to check that they belong to the boundary
of no $2$-cell.
We therefore here also have a surjective homomorphism
$\phi: G_{\cD} \to G_2 = F_2 \ltimes \ZZ/(2)$.
See \cite{marreiros} for far more on irregular disks.
\end{remark}

\section{Generators}
\label{sect:gen}

Our construction of $\cC_{\cD}$ yields
a finite but complicated presentation of $G_{\cD} = \pi_1(\cC_{\cD})$.
We present a more convenient family of generators of $G_{\cD}$,
which will be particularly useful in order to prove
that certain disks $\cD$ are regular.

A quadriculated disk $\cD$ is {\em hamiltonian}
if, seen as a graph, it is hamiltonian.
Notice that if a balanced quadriculated disk is hamiltonian
then it is tileable:
just place dominoes along a hamiltonian path.
A rectangle $\cD = [0,L] \times [0,M]$, $LM$ even,
is an example of a hamiltonian disk
(Figures \ref{fig:R2} and \ref{fig:rectangularpaths}
show examples of hamiltonian paths).

\begin{figure}[ht]
\begin{center}
\includegraphics[scale=0.275]{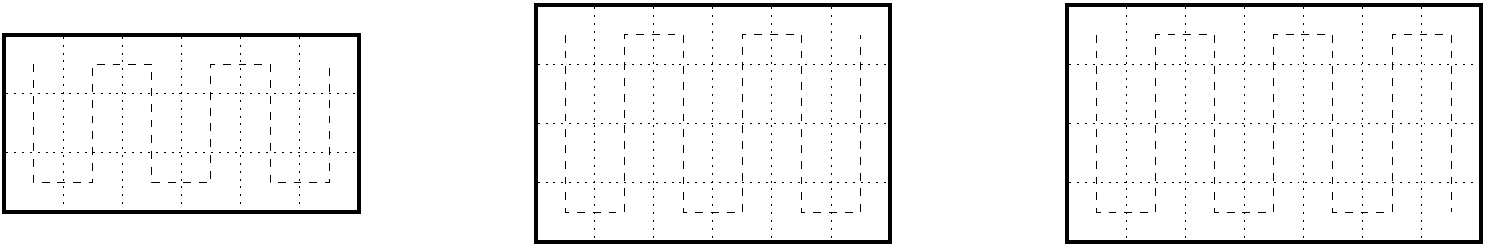}
\end{center}
\caption{Hamiltonian paths for rectangles $[0,L] \times [0,M]$.}
\label{fig:rectangularpaths}
\end{figure}

Given a hamiltonian quadriculated disk,
fix an arbitrary hamiltonian path
$\gamma_0 = (s_1, \ldots, s_{|\cD|})$
where the $s_i$ are distinct unit squares contained in $\cD$
and $s_i$ and $s_{i+1}$ are adjacent (for all $i$).
We say that a planar domino $d \subset \cD$
is contained in the path $\gamma_0$
if $d = s_i \cup s_{i+1}$ (for some $i$).
Similarly, we say that a horizontal domino $d \subset \cR_N$
{\em respects} $\gamma_0$ if its projection 
$\tilde d \subset \cD$ is a planar domino contained in $\gamma_0$;
by definition, vertical dominoes always {\em respect} $\gamma_0$.

For each plug $p$, construct a tiling
$\bt_p \in \cR_{N;p,\emptyplug}$, $N$ even, 
as in the proof of Lemma \ref{lemma:flipcork},
using $\gamma_0$ as spanning tree.
Notice that all dominoes in $\bt_p$ respect $\gamma_0$.
Again, Figure \ref{fig:R2} provides an example.
The following lemma shows that,
given the hamiltonian path $\gamma_0$ and a plug $p \in \cP$,
the tiling $\bt_p$ above is well defined up to flips.

\begin{lemma}
\label{lemma:welldefined}
Consider an hamiltonian quadriculated disk $\cD$ with a fixed path $\gamma_0$.
Consider a plug $p \in \cP$
and two tilings
$\bt_0 \in \cR_{N_0;p,\emptyplug}$ and
$\bt_1 \in \cR_{N_1;p,\emptyplug}$,
where $N_0$ and $N_1$ are both even.
If both $\bt_0$ and $\bt_1$ respect $\gamma_0$ then
$\bt_1^{-1} \ast \bt_0 \approx \bt_{\nvert} \in \cR_{N_0+N_1}$
and
$\bt_0 \ast \bt_1^{-1} \approx \bt_{\nvert} \in \cR_{N_0+N_1;p,p}$.
In particular, the paths in $\cC_{\cD}$ defined by $\bt_0$ and $\bt_1$
are homotopic with fixed endpoints.
\end{lemma}

\begin{proof}
Given a hamiltonian path $\gamma_0$ and
a cork $\cR_{N_i;p,\emptyplug}$,
we construct a planar region $\tilde\cD_p \subseteq [0,|\cD|] \times [0,N_i]$
and a folding map from $\tilde\cD_p$ to $\cR_{N_i;p,\emptyplug}$.
The folding map takes a unit square $[j-1,j]\times[k-1,k]$
to the unit cube $s_j \times [k-1,k]$.
Consistently, we obtain the quadriculated disk $\tilde\cD_p$
from $[0,|\cD|] \times [0,N_i]$
by removing the unit squares $[j-1,j]\times[0,1]$
for which $s_j \subset p$;
notice that $N_i \ge 2$ implies that $\tilde\cD_p$ is contractible.
A tiling $\bt_i$ of $\cR_{N_i;p,\emptyplug}$
respects $\gamma_0$ if and only if it can be unfolded,
i.e., if and only if it is the image under the folding map
of a tiling $\tilde\bt_i$ of $\tilde\cD_p$.
The result follows from the well known fact that
two domino tilings of a quadriculated disk
can be joined by a finite sequence of flips.
\end{proof}

Consider a hamiltonian disk $\cD$ with a fixed path $\gamma_0$.
Consider a floor $f_1 = (p_{0},f_1^{\ast},p_1)$;
let $f_{\nvert} = (p_1,\emptyset,p_1^{-1})$
be a matching vertical floor
(recall that $p_1^{-1}$ is the complement of $p_1$).
As above, construct tilings
\[ \bt_{p_0} \in \cR_{N_{p_0};{p_0},\emptyplug}, \quad
\bt_{p_1^{-1}} \in \cR_{N_{p_1^{-1}};p_1^{-1},\emptyplug}, \quad
\bt_{f_1} = \bt_{p_0}^{-1} \ast f_1 \ast f_{\nvert} \ast \bt_{p_1^{-1}}
\in \cR_{N} \]
where $N = N_{p_0}+N_{p_1^{-1}}+2$:
notice that the dominoes in the tiling $\bt_{f_1}$
which do not respect $\gamma_0$
are all contained in the original floor $f_1$.

\begin{lemma}
\label{lemma:floorsast}
Consider an hamiltonian quadriculated disk $\cD$ and a fixed path $\gamma_0$.
Consider $N$ even and a tiling $\bt \in \cR_N$
with floors $f_1, \ldots, f_N$;
we then have
\[ 
\bt \sim \bt_{f_1} \ast \bt_{f_2}^{-1} \ast \cdots 
\ast \bt_{f_i}^{(-1)^{(i+1)}} \ast \cdots \ast \bt_{f_N}^{-1}. 
\]
\end{lemma}

\begin{proof}
For each $i$, write $f_i = (p_{i-1},f_i^{\ast},p_i)$.
First, between each pair of floors $f_i$ and $f_{i+1}$
insert a large number of vertical floors.
Second, as in the proof of Lemma \ref{lemma:flipcork},
for even $i$,
modify the region between the original floors $f_i$ and $f_{i+1}$
from $\bt_{\nvert} \in \cT(\cR_{N_i;p_i,p_i})$
to $\bt_{p_i} \ast \bt_{p_i}^{-1}$.
Thus, after $f_i$ we now have an even number of floors
(all of whose dominoes respect $\gamma_0$),
followed by $\emptyplug$,
followed by the same even number of floors,
followed by $f_{i+1}$.
Similarly, for odd $i$,
repeat the construction between the original floors $f_i$ and $f_{i+1}$
but using $p_i^{-1}$ (instead of $p_i$);
we thus obtain the desired tiling and complete the proof.
\end{proof}

We consider a special case of the above construction.
Consider a domino $d \subset \cD$
which is not contained in the path $\gamma_0$;
thus, $d = s_i \cup s_j$ where $i +1 < j$.
Consider a plug $p \in \cP$ disjoint from $d$,
so that $\tilde p = p \cup d$ is a plug distinct from $p$.
Set $p_0 = p$, $p_1 = \tilde p^{-1}$ 
so that $\cD_{p_0,p_1}$ consists of the domino $d$, only.
Let $f_1 = (p_0,d,p_1)$ and
define $\bt_{d;p} = \bt_{f_1}$ as above.
Thus, the only domino of $\bt_{d;p}$
which does not respect $\gamma_0$ is $d \times [0,1]$.

Figure \ref{fig:tdp} below illustrates this construction
for $\cD = [0,4] \times [0,4]$:
we show a path $\gamma_0$, a domino $d$ not contained in $\gamma_0$,
a plug $p \in \cP$ and a valid tiling
$\bt_{d;p} \in \cT(\cR_{-2,2})$.
Notice that $\plug_0(\bt_{d;p}) = p$.
The domino $d \times [0,1]$ is the only one not respecting $\gamma_0$;
it appears in $\floorop_1(\bt_{d;p})$, the third floor in the figure.

\begin{figure}[ht]
\begin{center}
\includegraphics[scale=0.275]{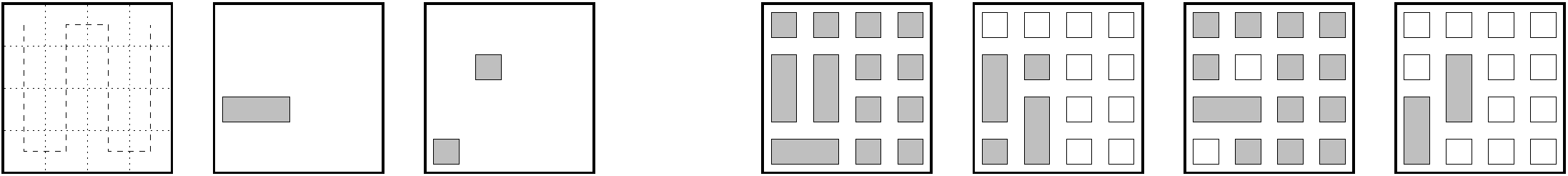}
\end{center}
\caption{For $\cD = [0,4] \times [0,4]$ we show a path $\gamma_0$,
a domino $d$, a plug $p$ and a valid tiling $\bt_{d;p}$.}
\label{fig:tdp}
\end{figure}

Notice that if the construction of $\bt_{d,p}$ above
is performed with a domino $d$
which is contained in the path $\gamma_0$
(and any compatible plug $p$)
we obtain a tiling $\bt_{d;p}$ such that every domino respects $\gamma_0$:
the tiling is therefore a tiling of $\gamma_0 \times [0,N]$
(for some positive even $N$)
and therefore $\bt_{d;p} \approx \bt_{\nvert}$.

\begin{lemma}
\label{lemma:decfloor}
Consider a hamiltonian disk $\cD$ with a fixed path $\gamma_0$.
Consider a floor $f = (p,f^{\ast},\tilde p)$.
Let $f^{\ast} = \{d_0, \ldots, d_{k-1}\}$ so that
$\cD_{p,\tilde p} = d_0 \cup \cdots \cup d_{k-1}$ and
$k = \frac12 |\cD_{p,\tilde p}|$.
Let $p_0 = p$ and $p_i = p_{i-1} \cup d_{i-1}$ so that $p_k = \tilde p^{-1}$.
Then 
\[ \bt_{f} \sim \bt_{d_0;p_0} \ast \cdots \ast
\bt_{d_i;p_i} \ast \cdots \ast \bt_{d_{k-1};p_{k-1}}. \]
\end{lemma}

\begin{proof}
Recall that
$\bt_{f} = \bt_{p}^{-1} \ast f \ast f_{\nvert} \ast \bt_{\tilde p^{-1}}$.
As in the proof of Lemma \ref{lemma:floorsast},
insert a large number of vertical floors around $f$.
As in Lemma \ref{lemma:movevert}, the horizontal dominoes
can be moved up or down:
do so so that they appear in the desired order,
with significant vertical space between them.
As in the proof of Lemma \ref{lemma:floorsast},
change the region between $d_i$ and $d_{i+1}$
to introduce $\bt_{p_{i+1}}^{-1} \ast \bt_{p_{i+1}}$.
This obtains the desired tiling and proves our lemma.
\end{proof}

Consider a hamiltonian disk $\cD$ with a fixed path $\gamma_0$;
assume without loss of generality that
the color of the square $s_i$ is $(-1)^i$.
Consider a domino $d = s_{i_{d,-}} \cup s_{i_{d,+}}$
not contained in $\gamma_0$
so that we may assume $i_{d,-} + 1 < i_{d,+}$.
We say that $d$ decomposes $\gamma_0$ into the three following intervals:
$I_{d;-1} = \ZZ \cap [1,i_{d,-}-1]$,
$I_{d;0} = \ZZ \cap [i_{d,-}+1,i_{d,+}-1]$,
$I_{d;+1} = \ZZ \cap [i_{d,+}+1,|\cD|]$.
Notice that the sets $I_{d;\pm 1}$ may be empty;
the interval $I_{d;0}$ always has even and positive cardinality.
Recall that a plug $p$ is compatible with $d$
if there is no unit square contained in both $p$ and $d$.
Consider a plug $p \in \cP$ compatible with $d$ and $j \in \{-1,0,+1\}$;
define
\[ \flux_j(d;p) = \sum_{i \in I_{d;j}, s_i \subset p} (-1)^i. \]
A verbal description may be helpful;
in order to compute $\flux_j(d;p)$ go through the list of squares
in  both $p$ and $I_{d;j}$:
each such square contributes with $+1$ or $-1$ according to color.
Notice that we always have 
$\flux_{-1}(d;p) + \flux_0(d;p) + \flux_{+1}(d;p) = 0$.
Define $\flux(d;p) =
(\flux_{-1}(d;p),\flux_0(d;p),\flux_{+1}(d;p)) \in H$
for $H = \{(\phi_{-1},\phi_0,\phi_{+1}) \in \ZZ^3 |
\phi_{-1}+\phi_0+\phi_{+1}=0 \}$.

For $d$, $p$ and $\bt_{d;p}$ as in Figure \ref{fig:tdp},
we have $d = s_3 \cup s_6$,
$I_{d;-1} = \{1,2\}$,
$I_{d;0} = \{4,5\}$,
$I_{d;+1} = \{7,\ldots, 16\}$,
$\flux(d;p) = (0,-1,+1)$.

\begin{lemma}
\label{lemma:flux}
Consider a hamiltonian disk $\cD$ with a fixed path $\gamma_0$
and a domino $d \subset \cD$ not contained in $\gamma_0$.
Consider two plugs $p_0, p_1 \in \cP$,
both compatible with $d$.
If $\flux(d;p_0) = \flux(d;p_1)$ then
$\bt_{d;p_0} \sim \bt_{d;p_1}$.
\end{lemma}

\begin{proof}
This proof requires familiarity with
the main results of \cite{saldanhatomei1995},
particularly the concept of flux for quadriculated surfaces
and  Theorem 4.1.

Assume without loss of generality that 
$\bt_{d;p_0}, \bt_{d;p_1} \in \cT(\cR_{-N,N})$;
we prove that $\bt_{d;p_0} \approx \bt_{d;p_1}$.
Indeed, both can be interpreted as tilings of 
the quadriculated surface
\[ (\gamma_0 \times [-N,N]) \smallsetminus
((s_{i_{d,-}} \times [0,1]) \cup
(s_{i_{d,+}} \times [0,1])). \]
Here we use folding as in the proof of Lemma \ref{lemma:welldefined},
so that the quadriculated region is a rectangle minus two unit squares.
The hypothesis  $\flux(d;p_0) = \flux(d;p_1)$ 
shows that the two tilings have the same flux
in the sense of \cite{saldanhatomei1995}.
It follows from Theorem 4.1 of \cite{saldanhatomei1995} that,
interpreted as tilings of this surface,
we have $\bt_{d;p_0} \approx \bt_{d;p_1}$.
The resulting sequence of flips is also good in $\cR_{-N,N}$,
proving the claim and completing the proof.
\end{proof}

\begin{remark}
\label{rem:twflux}
Consider a hamiltonian disk $\cD$ with a fixed path $\gamma_0$
and a domino $d \subset \cD$ not contained in $\gamma_0$.
Then there exists $s \in \{+1,-1\}$ such that
for every plug $p$ compatible with $d$ we have
\[ \Tw(\bt_{d;p}) = s \sum_j (-1)^j \flux_j(d;p). \]
\end{remark}

Given a hamiltonian disk $\cD$, a fixed path $\gamma_0$
and a domino $d \subset \cD$ not contained in $\gamma_0$,
let 
\begin{gather*}
\Phi_d = \{ (\phi_{-1},\phi_0,\phi_{+1}) \in H \;|\;
\forall j, \phi_j \in [\phi_j^{\min},\phi_j^{\max}] \}, \\
\phi_j^{\min} = - |\{i \in I_{d;j} | (-1)^i = -1\}|, \qquad
\phi_j^{\max} = |\{i \in I_{d;j} | (-1)^i = -1\}|.
\end{gather*}
Clearly, for all $p \in \cP$,
if $p$ is compatible with $d$ then $\flux(d;p) \in \Phi_d$;
conversely, for all $\phi \in \Phi_d$ there exists $p \in \cP$
such that $p$ compatible with $d$ and $\flux(d;p) = \phi$.
A {\em complete family} of compatible plugs for $d$
is a family $(p_{d,\phi})_{\phi \in \Phi_d}$
with $\flux(d;p_{d,\phi}) = \phi$ (for all $\phi \in \Phi_d$).

\begin{coro}
\label{coro:generators}
Consider a hamiltonian disk $\cD$ with a fixed path $\gamma_0$.
For each domino $d \subset \cD$ not contained in $\gamma_0$,
consider a complete family of compatible plugs
$(p_{d,\phi})_{\phi \in \Phi_d}$.
Consider the family of tilings $(\bt_{d;p_{d,\phi}})$:
this family of tilings generates the domino group $G^{+}_{\cD}$.
\end{coro}

\begin{proof}
This follows directly from Lemmas \ref{lemma:floorsast},
\ref{lemma:decfloor} and \ref{lemma:flux}.
\end{proof}


\section{Small regular rectangles}
\label{sect:44}

In this section we apply the results of the previous section,
particularly Corollary \ref{coro:generators},
to compute $G_{\cD}$ for a few examples.

\begin{lemma}
\label{lemma:44}
The rectangle $\cD = [0,4] \times [0,4]$ is regular.
Thus, $\Tw: G_{\cD}^{+} \to \ZZ$ is an isomorphism.
The group $G_{\cD}$ is isomorphic to $\ZZ \oplus (\ZZ/(2))$,
with generators $a$, the tiling shown in Figure \ref{fig:tdp},
and $c$, given by any tiling of $\cR_1$; 
$a \in G_{\cD}^{+}$ has twist $1$, $c$ has order $2$ and we have
$a\ast c \approx c\ast a$.
\end{lemma}

\begin{proof}
The proof is now a long computation.
The verification that $a\ast c \approx c\ast a$
is given by an explicit sequence of flips.
We apply  Corollary \ref{coro:generators}
to obtain a manageable list of generators of $G_{\cD}^{+}$;
for each generator $\bt$ we compute $k = \Tw(\bt)$
and verify (by an explicit sequence of flips)
that $\bt \sim a^k$.
 
We use the same path $\gamma_0$ shown in Figure \ref{fig:tdp}.
We first list all dominoes $d$ not contained in $\gamma_0$:
there are $9$ such dominoes, three dominoes per column,
each contained in a single row.
For each such domino $d$,
we list all the (finitely many) possible values
of $\flux(d;\ast) \in H \subset \ZZ^3$.
For instance, for $d$ as in Figure \ref{fig:tdp}
we have, for any $p \in \cP$,
$|\flux_{-1}(d;p)| \le 1$ and $|\flux_0(d;p)| \le 1$,
thus giving us a list of $9$ values.
For each such value we obtain an explicit $p$,
compute $\bt_{d;p}$ and complete the verification as above.
Notice that the tiling $\bt_{d;p}$ in Figure \ref{fig:tdp},
used to define $a$, is an instance of this construction.

This verification is performed by a computer,
but there are some simplifications which
significantly reduce the amount of computations involved.
For instance, the fact that $\cD$ is symmetric
with respect to a vertical line reduces from $9$ to $6$
the number of dominoes $d$ to be checked.
Also, for each $d$ it suffices to consider $\phi \in \Phi_d$
for which $|\phi_0|$ is maximal.
Indeed, assume for concreteness that $d$ is in the first column.
If $|\phi_0|$ does not have
the largest possible value (for that $d$)
then we may take $p$
which marks neither of the unit squares one row below $d$.
A few flips then take $\bt_{d,p}$ to $\bt_{\tilde d,p}$
where $\tilde d$ is the domino in the same column as $d$, one row lower;
this domino $\tilde d$ can be assumed to have already been taken care of.
After these simplifications,
there are less than $20$ distinct cases to be verified,
so the computer work can be double checked by hand.
\end{proof}

\begin{lemma}
\label{lemma:34}
Let $\cD = [0,L] \times [0,M]$ where
$L, M \in [3,6] \cap \ZZ$ and $LM$ is even:
the quadriculated disk $\cD$ is regular.
\end{lemma}

%

\begin{proof}
Again, a finite and manageable computation.
Construct explicit candidates for generators $a$ and $c$.
The tiling $a \in \cT(\cR_4)$ can be formed by taking
a copy of one of the tilings in Figure \ref{fig:234}
in a $2\times 3\times 4$ box and vertical dominoes elsewhere
(this is how we obtained $a$ for $\cD = [0,4]\times [0,4]$
in Figure \ref{fig:tdp} and in Lemma \ref{lemma:44}).
Notice that $\Tw(a) = +1$.
It is nice but not strictly necessary to check that
the choice of $\bt_0$ or $\bt_1$
and different positions for the box
obtain the same element of $G^{+}_{\cD}$.
For $c$, we take a tiling of $\cR_1$;
we verify that $a\ast c \sim c\ast a$.

For each quadriculated disk, choose a path $\gamma_0$ and
list the dominoes $d$ not contained in $\gamma_0$.
For each $d$, list the finitely many possible values of $\flux(d;\ast)$.
For each value of the flux, choose a plug $p$ and construct $\bt_{d;p}$.
For each such $\bt_{d;p}$, compute $k = \Tw(\bt_{d;p})$
and verify (by an explicit sequence of flips)
that $\bt_{d;p} \sim a^k$.
By Corollary \ref{coro:generators}, we are then done
(as in Lemma \ref{lemma:44}, after a relatively short computer verification).
\end{proof}

The proofs of Lemmas \ref{lemma:44} and \ref{lemma:34}
thus describe an algorithm.
Given a hamiltonian quadriculated disk $\cD$
containing a $2\times 3$ rectangle,
we fix a path $\gamma_0$ and
we construct $a \in \cR_4$ with $\Tw(a) = +1$
and vertical dominoes outside a $2\times 3\times 4$ box,
as in Figures \ref{fig:234} and \ref{fig:tdp}.
We make a list of the dominoes $d_i \subset \cD$
not contained in $\gamma_0$ and for each domino
we make a list of compatible plugs $p_j$
covering all possible values of $\flux(d_i;p_j)$.
For each pair $(d_i;p_j)$ we construct a tiling $\bt_{d_i;p_j}$
and compute $k = \Tw(\bt_{d_i;p_j})$.
We then obtain a finite list of questions of the form:
here are two tilings $\bt_{d_i,p_j}$ and $a^k$;
is it the case that $\bt_{d_i,p_j} \sim a^k$?
If the answer is {\em yes} in every case then 
this can be verified in finite time and 
we obtain a proof that $\cD$ is regular
(similar to the proofs of Lemmas \ref{lemma:44} and \ref{lemma:34}).
If in some case the answer is {\em no} then $\cD$ is not regular;
in order to prove that $\bt_0 \not\sim \bt_1$
we require a new idea or construction,
as was the case for Lemma \ref{lemma:thin}.

A natural question at this point is:
exactly which quadriculated disks are regular?
As of this writing we do not have a complete answer.
Figure \ref{fig:regularornot} shows five small quadriculated disks.
The first three are regular,
as can be verified using the methods described above.
The same methods applied to the last two are inconclusive,
but suggest that they are most likely not regular.

\begin{figure}[ht]
\begin{center}
\includegraphics[scale=0.275]{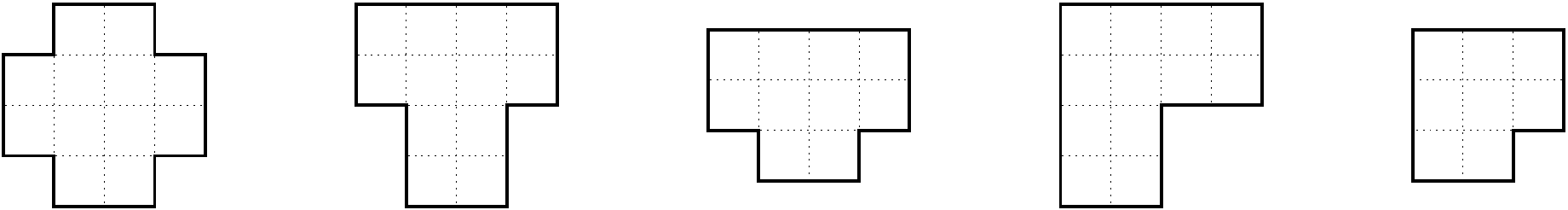}
\end{center}
\caption{Five quadriculated disks: which are regular?}
\label{fig:regularornot}
\end{figure}


\section{Larger regular rectangles}
\label{sect:thick}

In this section we move from specific quadriculated disks
to a large family of examples.

\begin{lemma}
\label{lemma:thick}
Let $\cD = [0,L] \times [0,M]$ where $L, M \ge 3$ and $LM$ is even:
the quadriculated disk $\cD$ is regular.
\end{lemma}

Together with Lemma \ref{lemma:thin},
this completes the proof of Theorem \ref{theo:rectangle}.
We first prove a sublemma.

\begin{lemma}
\label{lemma:thicksublemma}
Let $L \ge 3$ be a fixed number.
If $L$ is odd and $[0,L] \times [0,M]$ is regular
for both $M = 4$ and $M = 6$
then $[0,L] \times [0,M]$ is regular for any even $M > 6$.
If $L$ is even and $[0,L] \times [0,M]$ is regular
for all $M \in [3,6] \cap \ZZ$
then $[0,L] \times [0,M]$ is regular for any $M > 6$.
\end{lemma}

\begin{proof}
The proof is by induction on $M$.
We take the hamiltonian path $\gamma_0$
as indicated in Figure \ref{fig:rectangularpaths}.
We apply Corollary \ref{coro:generators}:
we must prove that for every domino $d \subset \cD$ and $\phi \in \Phi_d$
there exists a compatible plug $p$ for which $\flux(d;p) = \phi$ and
$\bt_{d;p} \sim a^k$ (for some $k \in \ZZ$).
Here again $a \in \cT(\cR_4)$ is a tiling
similar to the one shown in Figure \ref{fig:tdp}.

Consider first $d$ in the first column.
Clearly $|\phi_{-1}| + |\phi_0| < L$
for any $\phi \in \Phi_d$ and therefore also $|\phi_{+1}| < L$.
For any $\phi \in \Phi_d$ 
take $p = p_{\phi} \in \cP$ (with $\flux(d;p) = \phi$)
by marking in $I_{d;+1}$ the first $|\phi_{+1}|$ unit squares
of color $\sign(\phi_{+1})$
(as in the first example in Figure \ref{fig:thicksublemma}).
This implies that $p$ marks only unit squares
in the first four columns.
Let $\tilde\cD = [0,L] \times [0,4] \subseteq \cD$
be the subdisk formed by the first four columns.
Following the usual construction, all dominoes of $\bt_{d,p}$
outside $\tilde\cD \times [-N,N]$ are vertical with the same parity.
Let $\tilde\bt_{d,p}  \in \cT(\tilde\cD\times [-N,N])$ be the restriction
of $\bt_{d,p}$ to $\tilde\cD\times [-N,N]$.
By hypothesis, if $N$ is taken sufficiently large then
$\tilde\bt_{d,p} \approx a^k$, i.e.,
there exists a sequence of flips taking one tiling to the other.
By mere juxtaposition of vertical dominoes
in $(\cD \smallsetminus \tilde\cD) \times [-N,N]$
we have $\bt_{d,p} \approx a^k$,
proving this first case.
A similar argument holds if $d$ is in the last column.

\begin{figure}[ht]
\begin{center}
\includegraphics[scale=0.275]{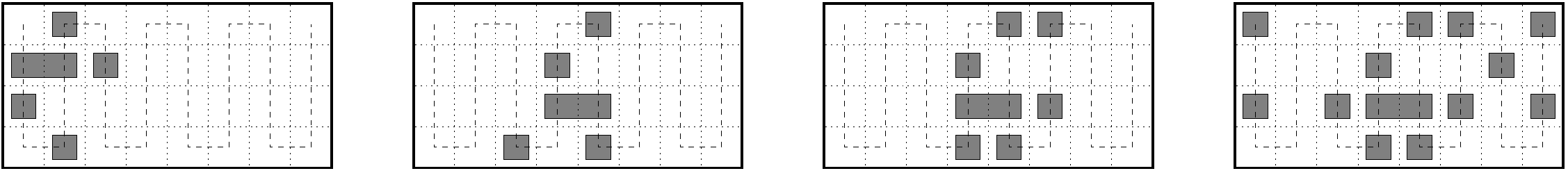}
\end{center}
\caption{A rectangle $\cD = [0,L] \times [0,M]$,
a path $\gamma_0$ and a few examples of pairs $(d,p)$;
$d$ is a domino not contained in $\gamma_0$ and $p$ is a compatible plug.}
\label{fig:thicksublemma}
\end{figure}

Consider now $d$ in some intermediate column
(neither the first nor the last).
We first consider the subcase where 
$\phi_{-1} \phi_{+1} \ge 0$.
We clearly have $|\phi_0| < L$
and therefore also  $|\phi_{-1}| + |\phi_{+1}| < L$.
Construct $p = p_{\phi}$ by selecting squares in $I_{d;\pm 1}$
as near as possible to the columns occupied by $d$
(as in the second example in Figure \ref{fig:thicksublemma}).
At most $6$ columns are occupied:
take $\tilde\cD \subset \cD$ to be the union of the occupied columns.
By hypothesis, $\tilde\cD$ is regular;
the proof proceeds as in the previous case.

Consider finally the subcase where $d$ is in some intermediate column
and $\phi_{-1} \phi_{+1} < 0$.
Assume $\phi_{-1} < 0$ (the other case is similar).
Let $l \le \lceil \frac{L}{2} \rceil$ be
the number of unit squares of color $-1$ in the first column of $\cD$.
Thus, if $L$ is even then $l = \frac{L}{2}$;
if $L$ is odd then either $l = \frac{L-1}{2}$ or $l = \frac{L+1}{2}$.
Notice that $l$ is also
the number of unit squares of color $+1$ in the last column of $\cD$.
If either $|\phi_{-1}| < l$ or $|\phi_{+1}| < l$ then
$p = p_{\phi}$ can be constructed, as in the previous cases,
so as to occupy at most $6$ columns and the proof proceeds as before
(as in the third example in Figure \ref{fig:thicksublemma}).
We may thus assume $|\phi_{-1}| \ge l$ and $|\phi_{+1}| \ge l$.

Construct $p = p_\phi$ marking all unit squares of color $-1$ 
in the first column and all unit squares of color $+1$ in the last column
(as in the fourth example in Figure \ref{fig:thicksublemma}).
Let $\tilde\cD = [0,L] \times [1,M-1] \subset \cD$;
$\tilde\cD$ is regular by induction hypothesis.
Let $\tilde\phi = (\phi_{-1} + l, \phi_0, \phi_{+1} - l)$
and let $\tilde p$ be the plug for $\tilde\cD$
obtained from $p$ by intersection, i.e.,
by discarding the $l$ marked squares in the first column
and the $l$ marked squares in the last column.
We have $\flux(\tilde p) = \tilde\phi$.
Construct as usual the tiling
$\tilde\bt_{d;\tilde p} \in \cT(\tilde\cD \times [-\tilde N,\tilde N])$.
By induction hypothesis, we have
$\tilde\bt_{d;\tilde p} \approx a^k$
provided $\tilde N$ is taken large enough (and even);
here $k = \Tw(\tilde\bt_{d;\tilde p})$.
We may also assume that $a$ occupies the last two rows
and three central columns of $\tilde\cD$,
thus leaving free at least the first row,
the first and last column of $\tilde\cD$
(we assume here $M > 6$, as we can).

Construct $\bt_0 = \bt_{d;p} \in \cT(\cD \times [-N,N])$ as usual,
matching the $l$ squares in the first column
with the $l$ squares in the last column:
these are the last matches to be addressed,
and we may leave vertical space $[-\tilde N,\tilde N]$
for the previous matches,
so that $N = \tilde N + 2l$.
Notice that $\bt_0$ respects the subregion
$\tilde\cD \times [-\tilde N,\tilde N]$
and coincides there with $\tilde\bt_{d;\tilde p}$.
We thus have $\bt_0 \approx \bt_1$,
where $\bt_1 \in \cT(\cD \times [-N,N])$ 
coincides with $a^k$ in $\tilde\cD \times [-\tilde N,\tilde N]$
and with $\bt_0$ elsewhere.
We are therefore left with proving that
$\bt_1 \approx \bt_2$
where $\bt_2 \in \cT(\cD \times [-N,N])$ 
coincides with $a^k$ (and with $\bt_1$)
in $\tilde\cD \times [-\tilde N,\tilde N]$
and is vertical elsewhere.

\begin{figure}[ht]
\begin{center}
\includegraphics[scale=0.275]{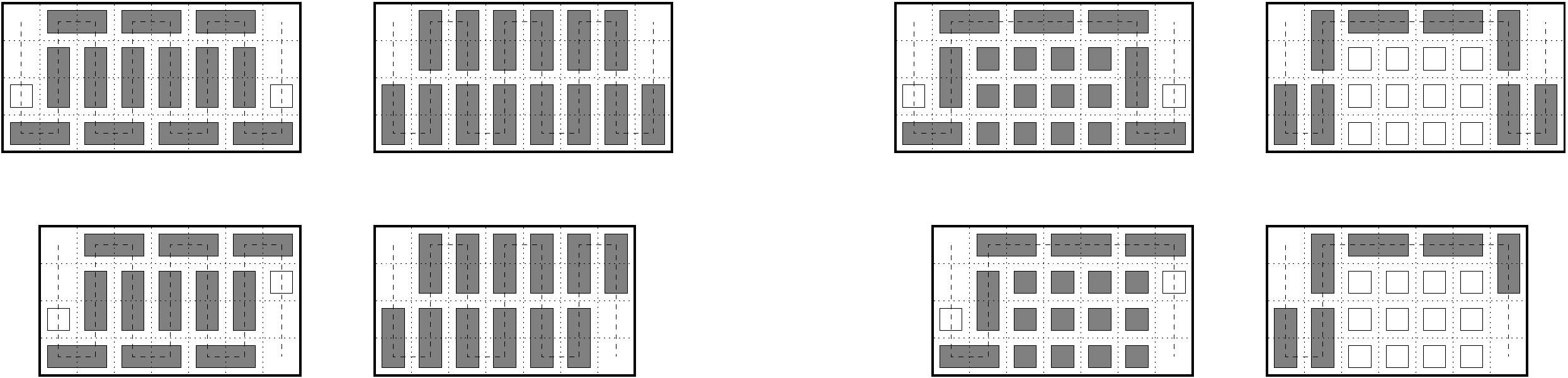}
\end{center}
\caption{The path $\gamma_1$ and the floors $[\tilde N, \tilde N+2]$
for the tilings $\bt_1$ and $\bt_3$.}
\label{fig:penultimate}
\end{figure}

We construct a new tiling $\bt_3 \approx \bt_1$;
the floors $[-\tilde N,\tilde N]$ of $\bt_1$ and $\bt_3$ coincide.
In order to construct the remaining floors of $\bt_3$,
first construct a path $\gamma_1$ coinciding with $\gamma_0$
in the first and last column of $\cD$
such that its intersection with $\tilde\cD$
is contained in the union of the first row
and the first and last columns of $\tilde\cD$,
as illustrated right half of Figure \ref{fig:penultimate}.
The floors $[-N,-\tilde N] \cup [\tilde N,N]$
of the tilings $\bt_0$ and $\bt_1$ are constructed
as in the proof of Lemma \ref{lemma:flipcork} and Figure \ref{fig:R2},
using the original path $\gamma_0$:
Figure \ref{fig:penultimate} shows floors $[\tilde N, \tilde N+2]$.
The tiling $\bt_3$ is similarly constructed,
but using, for the new floors, the path $\gamma_1$ instead.
The fact that $\bt_3 \approx \bt_1$ is proved
looking at pairs of floors $[2z,2z+2]$ (with $z \in \ZZ$),
as in Figure \ref{fig:penultimate};
we may either give an explicit sequence of flips
or use the results from \cite{primeiroartigo}.

Finally, we prove that $\bt_2 \approx \bt_3$.
Indeed, they coincide by construction outside the quadriculated surface
$\gamma_1 \times [-N,N]$,
which is respected by both tilings.
Thus, the problem of finding a sequence of flips from $\bt_2$ to $\bt_3$
is the problem of connecting by flips two rather explicit tilings 
of a quadriculated disk:
this follows either from an explicit sequence of flips
or from \cite{thurston1990} and \cite{saldanhatomei1995}.
\end{proof}

\begin{proof}[Proof of Lemma \ref{lemma:thick}]
Apply Lemma \ref{lemma:thicksublemma} for each $L \in [3,6] \cap \ZZ$:
the hypothesis is provided by Lemma \ref{lemma:34}.
We thus have that $[0,L] \times [0,M]$ is regular
provided $LM$ is even, $3 \le L \le 6$ and $M \ge 3$.
Or, equivalently,
provided $LM$ is even, $L \ge 3$ and $3 \le M \le 6$.
Apply Lemma \ref{lemma:thicksublemma} again for each $L \ge 3$
to obtain the desired conclusion.
\end{proof}

It would of course be interesting to prove that a larger class
of disks is regular.
As mentioned in Section \ref{sect:44}, regular disks seem to be common.


\section{The constant $c_{\cD}$ and
the spine $\tilde\cC^{\bullet}_{\cD}$}
\label{sect:cD}

Let $\cD$ be a fixed but arbitrary non trivial regular disk.
Let $\cC_{\cD}$ be the $2$-complex constructed
in Section \ref{sect:groupcomplex}.
Let $\Pi^{+}: \cC^{+}_{\cD} \to \cC_{\cD}$ be the double cover
constructed at the end of Section \ref{sect:groupcomplex}.
Since $\cD$ is regular, we have that
$\Tw: \pi_1(\cC^{+}_{\cD}) \to \ZZ$ is an isomorphism.
Let $\Pi: \tilde\cC_{\cD} \to \cC^{+}_{\cD}$ be its universal cover.
Let $\tilde\cP$ be the set of vertices of $\tilde\cC_{\cD}$;
select as base point a vertex $\tilde\emptyplug \in \tilde\cP$
which is a preimage of $\emptyplug \in \cP$.
Notice that the set of preimages in $\tilde\cP$ 
of the empty plug $\emptyplug \in \cP$ is naturally identified
with $\pi_1(\cC_{\cD}) \approx \ZZ \oplus \ZZ/(2)$.

As in Remark \ref{remark:cocycle},
lift $\tau^u \in C^1(\cC_{\cD};\RR)$ to
$\tau^u \in C^1(\tilde\cC_{\cD};\RR)$
(here $u \in \{e_1,e_2\}$ is fixed but arbitrary).
Since $\tilde\cC_{\cD}$ is simply connected
we have $\tau^u \in B^1(\tilde\cC_{\cD};\RR)$.
Integrate $\tau^u$ to obtain a function
$\tw: \tilde\cP \to \frac14\ZZ \subset \RR$
satisfying $\tw(\tilde\emptyplug) = 0$.
Notice that if $\bt \in \cT(\cR_N)$
is interpreted as a path from $[0,N]$ to $\cC_{\cD}$
then such a path can be lifted to $\tilde\bt: [0,N] \to \tilde\cC_{\cD}$
and we then have $\Tw(\bt) = \tw(\tilde\bt(N)) - \tw(\tilde\bt(0))$.
Let $\sigma: \tilde\cC_{\cD} \to \tilde\cC_{\cD}$
be a generator of the group of deck transformations
of the covering map $\Pi: \tilde\cC_{\cD} \to \cC^{+}_{\cD}$:
choose $\sigma$ such that $\tw(\sigma(p)) = 1+\tw(p)$
for all $p \in \tilde\cP$.

We are interested in {\em tiling paths}:
paths $\Gamma: [N_0,N_1] \to \cC_{\cD}$
(for $N_0, N_1 \in \ZZ$) taking integers to vertices (i.e., plugs)
and intervals $[j-j,j]$ (for $j \in \ZZ$) to edges (i.e., floors).
Such paths can of course be lifted to
$\tilde\Gamma: [N_0,N_1] \to \tilde\cC_{\cD}$
with $\Pi^{+} \circ \Pi \circ \tilde\Gamma = \Gamma$.
The lift is well defined if a starting point $\tilde\Gamma(N_0)$ is given.
As discussed in Section \ref{sect:groupcomplex},
tiling paths correspond to tilings of corks
$\cR_{N_0,N_0,p_0,p_1}$ where $p_0 = \Gamma(N_0)$ and $p_1 = \Gamma(N_1)$.
Consistently, write
$\Tw(\Gamma) = \tw(\tilde\Gamma(N_1)) - \tw(\tilde\Gamma(N_0))$.

\begin{figure}[t]
\begin{center}
\includegraphics[scale=0.275]{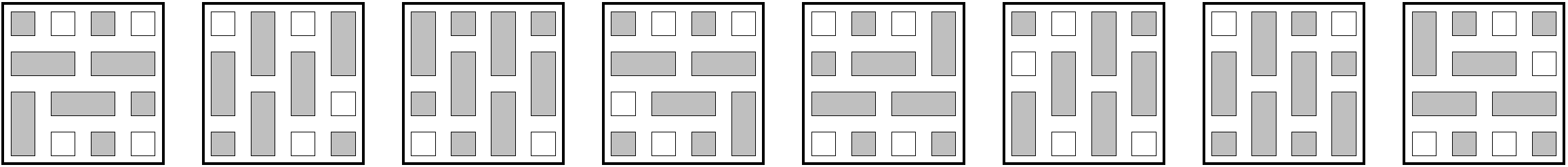}
\end{center}
\caption{A simple closed tiling path $\Gamma: [0,8] \to \cC_{\cD}$
for $\cD = [0,4]^2$; $\Tw(\Gamma) = 12$. }
\label{fig:rocket}
\end{figure}

A tiling path $\Gamma$ is {\em closed} if $\Gamma(N_0) = \Gamma(N_1)$;
Figure \ref{fig:rocket} shows an example of a closed tiling path
$\Gamma: [0,8] \to \cC_{\cD}$ for $\cD = [0,4]^2$.
A closed tiling path $\Gamma: [0,N] \to \cC_{\cD}$ is {\em simple}
if $k_0, k_1 \in \ZZ \cap [0,N)$ and $\Gamma(k_0) = \Gamma(k_1)$
imply $k_0 = k_1$.
Clearly, if $\Gamma: [0,N] \to \cC_{\cD}$
is a simple closed tiling path then $N \le |\cP|$.
Let $c_{\cD} \in \QQ \cap (0,+\infty)$
be the maximum value of $\Tw(\Gamma)/N$
taken over all simple closed tiling paths $\Gamma: [0,N] \to \cC_{\cD}$;
let $\Gamma_{\bullet}: [0,N_{\bullet}] \to \cC_{\cD}$
be a closed tiling path for which the maximum value is acheived.
Since we are taking the maximum over a non empty finite set,
the maximum is well defined.
Lemma \ref{lemma:cD} below provides
alternative characterizations of $c_{\cD}$.
Lemma \ref{lemma:tropical} and Example \ref{example:cD44}
show how to compute $c_{\cD}$
for a given regular quadriculated disk $\cD$.

For $p_0, p_1 \in \cP$ and $N \in \NN$,
let $m_{N;p_0,p_1} \in \{-\infty\} \cup \frac14 \ZZ$
be the maximum value of $\Tw(\Gamma)$
for $\Gamma: [0,N] \to \cC_{\cD}$ a tiling path with 
$\Gamma(0) = p_0$, $\Gamma(N) = p_1$.
We follow here the convention that the maximum of the empty set
is $-\infty$.
Thus, for instance, if $\cD_{p_0,p_1}$ admits no tiling
(in particular, if $p_0$ and $p_1$ are not disjoint)
then $m_{1;p_0,p_1} = -\infty$.
Even more degenerately,
$m_{0,p_0,p_1}$ equals $0$ if $p_0 = p_1$ and $-\infty$ otherwise.
The following result provides us with estimates for $m_{N,p_0,p_1}$.
We will further discuss these numbers below;
see Equation \ref{equation:tropicalpower} and Lemma \ref{lemma:tropical}.

\begin{lemma}
\label{lemma:cD}
Let $\cD$ be a fixed but arbitrary non trivial regular disk.
Let $c_{\cD} \in \QQ \cap (0,+\infty)$ be as defined above.
\begin{enumerate}
\item{For any closed tiling path $\Gamma: [0,N] \to \cC_{\cD}$
we have $|\Tw(\Gamma)| \le c_{\cD} N$.}
\item{There exists constants $d_{-}, d_{+} \in \RR$
(depending on $\cD$ only; not depending on $N$, $p_0$ or $p_1$) such that,
for all $p_0, p_1 \in \cP$ and all $N \ge 4|\cD|$,
\[ c_{\cD} N + d_{-} \le m_{N;p_0,p_1} \le c_{\cD} N + d_{+}. \]}
\end{enumerate}
\end{lemma}

\begin{proof}
A {\em double point} for a closed tiling path $\Gamma: [0,N] \to \cC_{\cD}$
is a pair $\{k_0,k_1\} \subset \ZZ \cap [0,N)$ 
with $k_0 < k_1$ and $\Gamma(k_0) = \Gamma(k_1)$.
We prove the first item by induction on the number of double points:
if there are $0$ double points the curve is simple
and the claim holds by definition of $c_{\cD}$.
Let $\{k_0,k_1\}$ be a double point.
Let $\Gamma_1: [0,k_1-k_0] \to \cC_{\cD}$
and $\Gamma_2: [0,N-k_1+k_0] \to \cC_{\cD}$
be closed tiling paths defined by
$\Gamma_1(k) = \Gamma(k_0+k)$ and
$\Gamma_2(k) = \Gamma(k_1+k)$;
if $k_1+k > N$ we interpret $\Gamma(k_1+k) = \Gamma(k_1+k-N)$.
By induction hypothesis we have
$|\Tw(\Gamma_1)| \le c_{\cD} (k1-k0)$ and
$|\Tw(\Gamma_2)| \le c_{\cD} (N-k1+k0)$.
From the definition of $\Tw$ we have
$\Tw(\Gamma) = \Tw(\Gamma_1) + \Tw(\Gamma_2)$.
We thus have
$|\Tw(\Gamma)| \le |\Tw(\Gamma_1)| + |\Tw(\Gamma_2)| \le c_{\cD} N$,
completing the proof of the first item.

Recall that
$\Gamma_\bullet: [0,N_\bullet] \to \cC_{\cD}$
is a simple closed tiling path
satisfying $\Tw(\Gamma_\bullet) = c_{\cD} N_\bullet$.
Let $\bt_{\bullet} \in \cT(\cR_{0,N_\bullet;p_{\bullet},p_{\bullet}})$
be the corresponding tiling where $p_\bullet = \Gamma_\bullet(0) \in \cP$.
Let $\bt_0 \in \cT(\cR_{0,4|\cD|;p_0,p_\bullet})$
be an arbitrary tiling
(its existence is guaranteed by Lemma \ref{lemma:cork}).
Similarly,
for each $k \in [4|\cD|,4|\cD|+N_\bullet] \cap \ZZ$,
let $\bt_k \in \cT(\cR_{0,k;p_\bullet,p_1})$
be an arbitrary tiling.
Let $\tilde d$ be the minimum value of $\Tw(\bt_0) + \Tw(\bt_k)$.
For $N \ge 8|\cD|$, consider the tiling
$\bt = \bt_0 \ast \bt_{\bullet}^j \ast \bt_k \in \cT(\cR_{0,N;p_0,p_1})$,
where $k = 4|\cD| + ( (N-8|\cD|) \bmod N_{\bullet} )$
and $j = \lfloor (N-8|\cD|)/N_{\bullet} \rfloor$.
In particular, $jN_\bullet \ge N - 8|\cD| - N_{\bullet}$.
We have 
\[ \Tw(\bt) = \Tw(\bt_0) + c_{\cD} jN_{\bullet} + \Tw(\bt_k)  
\ge c_{\cD} N - c_{\cD} (8|\cD| + N_{\bullet}) + \tilde d, \]
obtaining the desired $d_{-}$.

For each pair $(p_0,p_1) \in \cP^2$ let
$\bt_{p_1,p_0} \in \cT(\cR_{0,4|\cD|;p_1,p_0})$ be an arbitrary tiling.
For any $\bt \in \cT(\cR_{0,N;p_0,p_1})$ we have from the first item that
$\Tw(\bt \ast \bt_{p_1,p_0}) \le c_{\cD} (N+4|\cD|)$
and therefore
\[ \Tw(\bt) \le c_{\cD} N + 4c_{\cD} |\cD| - \Tw(\bt_{p_1,p_0}); \]
taking the minimum of $\Tw(\bt_{p_1,p_0})$ over all $(p_0,p_1) \in \cP^2$
gives us the desired $d_{+}$ and completes the proof.
\end{proof}

We briefly describe how to compute $c_{\cD}$.
In brief, the problem of computing $c_{\cD}$ is
the problem of computing an eigenvalue of a matrix,
but in the tropical semifield.
The subject of tropical mathematics is vast,
and we do not assume any knowledge of it;
\cite{litvinov} is a nice introductory text,
with an ample bibliography.
A semifield is a an algebraic structure
of the form $(A,1,\oplus,\otimes,\cdot^{-1})$ where $A$ is a set,
$1 \in A$ and the binary operations $\oplus$ and $\otimes$ in $A$
are associative and commutative, satisfy the usual distributive law,
and $\otimes$, together with $1$ and $\cdot^{-1}$,
endow $A$ with an abelian multiplicative group structure.
An obvious example is $A = (0,+\infty) \subset \RR$
with the usual operations.
The set $\frac14 \ZZ$ is a semifield with the operations
$a \oplus b = \max\{a,b\}$ and $a \otimes b = a+b$:
this is the {\em tropical} semifield.

If $\bM^N$ is a $\cP \times \cP$ matrix with entries
$(\bM^N)_{p_0,p_1} = m_{N;p_0,p_1}$ (as in Lemma \ref{lemma:cD})
then $\bM^N$ is the tropical power of $\bM = \bM^1$:
\begin{equation}
\label{equation:tropicalpower}
(\bM^{N_0+N_1})_{(p_0,p_2)} =
\max_{p_1 \in \cP}\;
\left((\bM^{N_0})_{(p_0,p_1)} + (\bM^{N_1})_{(p_1,p_2)}\right). 
\end{equation}
For a relatively small disk $\cD$ 
it is therefore not hard to compute $\bM^N$.
In this tropical context, a diagonal matrix $D$ has diagonal entries
in $\frac14 \ZZ$ and off-diagonal matrices equal to $-\infty$.
The inverse $D^{-1}$ has diagonal entries
$(D^{-1})_{p,p} = -D_{p,p}$
and conjugation $D^{-1} \bM^N D$ is of course defined
with tropical operations:
\[ (D^{-1}\bM^N D)_{p_0,p_1} = -D_{p_0,p_0} + (\bM^N)_{p_0,p_1} + D_{p_1,p_1}. \]
The number $c_{\cD}$ is (in this context) an eigenvalue of $\bM$.

\begin{lemma}
\label{lemma:tropical}
Let $\cD$ be a regular quadriculated disk and
let $\bM$ be the tropical matrix constructed above.
Let $D$ be a diagonal matrix and $N \in \NN^\ast$; then
\begin{equation}
\label{equation:upperc}
c_{\cD} \le \frac{m}{N}, \qquad
m =  \max_{p_0,p_1 \in \cP} (D^{-1}\bM^N D)_{p_0,p_1}.
\end{equation}
\end{lemma}


\begin{proof}
It follows from Equation \ref{equation:tropicalpower}
(and induction) that
\[ m_{kN;p_0,p_1} = (\bM^{kN})_{p_0,p_1} \le km + D_{p_0,p_0} - D_{p_1,p_1} \]
for all $k \in \NN^\ast$ (and for all $p_0, p_1 \in \cP$).
From Lemma \ref{lemma:cD}, $c_{\cD} \le m/N$.
\end{proof}

\begin{example}
\label{example:cD44}
For $\cD = [0,4]^2$, we have $c_{\cD} = \frac32$.
Indeed, the closed tiling path
shown in Figure \ref{fig:rocket} implies $c_{\cD} \ge \frac32$.
On the other hand,
for $N = 4$ it is not too hard to obtain a diagonal matrix $D$
for which $m = 6$
(where $m$ is defined in Equation \ref{equation:upperc},
Lemma \ref{lemma:tropical}):
we therefore have $c_{\cD} \le \frac32$.
In this example therefore we can take $N_{\bullet} = 8$
and $\Gamma_{\bullet}: [0,N_{\bullet}] \to \cC_{\cD}$
as in Figure \ref{fig:rocket}.
\end{example}



Recall that $\Gamma_{\bullet}$ is a simple closed tiling path
with $\Tw(\Gamma_{\bullet}) = c_{\cD} N_{\bullet}$.
Extend it to define
$\Gamma_{\bullet}: \RR \to \cC_{\cD}$,
periodic with period $N_{\bullet}$.
Lift this path to define
$\tilde\Gamma_{\bullet}: \RR \to \tilde\cC_{\cD}$
with $\tilde\Gamma_{\bullet}(t+N_{\bullet}) =
\sigma^m(\tilde\Gamma_{\bullet}(t))$
where $m = c_{\cD} N_{\bullet}$.
(Recall that $\sigma: \tilde\cC_{\cD} \to \tilde\cC_{\cD}$
is a deck transformation.)
The {\em spine} of $\tilde\cC_{\cD}$
is $\tilde\cC_{\cD}^{\bullet} \subset \tilde\cC_{\cD}$,
the image of $\tilde\Gamma_{\bullet}$.
Clearly, the spine $\tilde\cC_{\cD}^{\bullet}$ is
a $1$-complex isomorphic to $\RR$, with integers being vertices.

Endow both $\tilde\cC_{\cD}^{\bullet}$ and $\tilde\cC_{\cD}$
with metric structures:
each edge has length $1$
and the distance between two vertices $p_0$ and $p_1$
is the minimal $N$ for which there exists a tiling path $\Gamma$
with $\Gamma(0) = p_0$ and $\Gamma(N) = p_1$.
For each vertex $p$ of $\tilde\cC_{\cD}$,
let $\Pi(p) \in \tilde\cC_{\cD}^{\bullet}$ be the vertex of 
$\tilde\cC_{\cD}^{\bullet}$ nearest to $p$;
in case of a draw choose arbitrarily but preserve
$\Pi(\sigma^m(p)) = \sigma^m(\Pi(p))$.
Extend $\Pi$ to $1$ and $2$-cells,
always preserving the identity $\Pi \circ \sigma^m = \sigma^m \circ \Pi$
and the fact that the restriction of $\Pi$ to $\tilde\cC_{\cD}^{\bullet}$
is the identity.
Thus, $i: \tilde\cC_{\cD}^{\bullet} \to \tilde\cC_{\cD}$
and $\Pi: \tilde\cC_{\cD} \to \tilde\cC_{\cD}^{\bullet}$
are continuous maps taking vertices to vertices.
Clearly $\Pi \circ i$ is the identity map.

The following result shows that $i$ and $\Pi$
are quasi-isometries in the sense of Gromov.
There is a vast literature on quasi-isometries;
see for instance \cite{gromov} and \cite{sherdaverman}.
We will keep the discussion self contained.

\begin{lemma}
\label{lemma:quasi}
If $p_0, p_1$ are vertices of the spine $\tilde\cC_{\cD}^{\bullet}$
then
$d_{\tilde\cC_{\cD}^{\bullet}}(p_0,p_1) = d_{\tilde\cC_{\cD}}(p_0,p_1)$.
There exists a constant $d$ such that
$d_{\tilde\cC_{\cD}}(p,\Pi(p)) \le d$
for every vertex $p$ of $\tilde\cC_{\cD}$.
\end{lemma}

\begin{proof}
Let $p_0, p_1$ be vertices of the spine $\tilde\cC_{\cD}^{\bullet}$;
we may assume without loss of generality that $\Tw(p_0) < \Tw(p_1)$.
Clearly
$d_{\tilde\cC_{\cD}}(p_0,p_1) \le
d_{\tilde\cC_{\cD}^{\bullet}}(p_0,p_1)$;
assume by contradiction that 
$N_0 = d_{\tilde\cC_{\cD}}(p_0,p_1) <
N_1 = d_{\tilde\cC_{\cD}^{\bullet}}(p_0,p_1)$.
Take $\tilde N_1 > N_1$, $\tilde N_1 = kN_{\bullet}$, $k \in \NN^\ast$;
set $\tilde N_0 = N_0 + \tilde N_1 - N_1$.
Let $\Gamma: [0,N_0] \to \tilde\cC_{\cD}$
be a tiling path with $\Gamma(0) = p_0$, $\Gamma(N_0) = p_1$.
Extend $\Gamma$ to $\Gamma: [0,\tilde N_0] \to \tilde\cC_{\cD}$
by following the spine $\tilde\cC_{\cD}^{\bullet}$ in the positive direction
so that $\Gamma(\tilde N_0) = \sigma^{km}(p_0)$.
The closed curve $\Gamma: [0,\tilde N_0] \to \cC_{\cD}$
satisfies $\Tw(\Gamma) = km > c_{\cD}\tilde N_0$,
violating Lemma \ref{lemma:cD}.

For the second claim, consider the equivalence class in $\tilde\cP$
identifying $p$ with $\sigma^m(p)$.
There are finitely many equivalence classes:
let $X \subset \tilde\cP$ be a set of representatives.
Take $d = \max_{p \in X} d_{\tilde\cC_{\cD}}(p,\Pi(p))$:
the claim follows from invariance under $\sigma^m$.
\end{proof}


\section{Proof of Theorem \ref{theo:M}}
\label{sect:theoM}

Let $\Gamma_{\bullet}: [0,N_{\bullet}] \to \tilde\cC_{\cD}$
and the spine $\cC^{\bullet}_{\cD}$ be as in the previous section.
Let $d$ be as in Lemma \ref{lemma:quasi}.

Recall that $\tilde\cC_{\cD}$ is
the universal cover of the finite complex $\cC_{\cD}$
and therefore simply connected.
Given two tiling paths
$\Gamma_0: [0,N_0] \to \tilde\cC_{\cD}$ and
$\Gamma_1: [0,N_1] \to \tilde\cC_{\cD}$
with $\Gamma_0(0) = \Gamma_1(0)$ and $\Gamma_0(N_0) = \Gamma_1(N_1)$
there exists therefore a homotopy with fixed endpoints
between $\Gamma_0$ and $\Gamma_1$.
We may combinatorialize the concept of homotopy (with fixed endpoints)
as follows:
the homotopy is a family $(\Gamma_s)_{s \in \frac1S\ZZ \cap [0,1]}$,
of tiling paths ($S \in \NN^\ast$),
all with the same endpoints,
where two consecutive tiling paths
$\Gamma_{\frac{k}{S}}$ and $\Gamma_{\frac{k+1}{S}}$
differ by one of the two moves below.
\begin{enumerate}
\item{The two consecutive paths may have the same length
and differ by a flip, in other words, by moving the path across
one of the $2$-cells shown in Figures \ref{fig:hflip} and \ref{fig:vflip}.}
\item{The two consecutive paths may have lengths differing by two;
the longer one is obtained from the shorter one by inserting
two adjacent floors (edges), one the inverse of the other.}
\end{enumerate}

Two tilings $\bt_0, \bt_1 \in \cT(\cR_N)$ with $\Tw(\bt_0) = \Tw(\bt_1) = t$
thus correspond to two tiling paths
$\Gamma_0, \Gamma_1: [0,N] \to \tilde\cC_{\cD}$
with the same endpoints:
$\Gamma_0(0) = \Gamma_1(0) = \emptyplug$,
$\Gamma_0(N) = \Gamma_1(N) = \sigma^t(\emptyplug)$.
We know that a homotopy (with fixed endpoints) exists.
If all paths in the homotopy have length at most $N + M$
then $\bt_0 \ast \bt_{\nvert,M} \approx \bt_1 \ast \bt_{\nvert,M}$.
In order to prove Theorem \ref{theo:M}, therefore,
we need to control the length of paths in a homotopy.

\begin{proof}[Proof of Theorem \ref{theo:M}]
Recall that $d$ is as in Lemma \ref{lemma:quasi}.
Let $\tilde M$ be such that
if $\Gamma_0, \Gamma_1$ are tiling paths
with the same endpoints in $\tilde\cC_{\cD}$
with lengths $N_0, N_1 \le 4d + 4$
then there exists a homotopy (with fixed endpoints)
from $\Gamma_0$ to $\Gamma_1$ such that all intermediate paths
have length at most $\tilde M$.
The existence of such $\tilde M$ follows from the fact that
the number of such pairs of paths with values in $\cC_{\cD}$ is finite;
pairs differing by a deck transformation are equivalent.
We claim that $M = \tilde M + 4d$ satisfies 
the statement of the theorem.

Given an integer $t$ we construct a tiling path $\Gamma_0$
from $\emptyplug \in \tilde\cP$ to $\sigma^t(\emptyplug)$ as follows.
First construct the shortest arc from
$\emptyplug$ to the spine $\cC^{\bullet}_{\cD}$:
this will be the beginning of $\Gamma_0$.
Next construct the shortest arc from
the spine $\cC^{\bullet}_{\cD}$ to $\sigma^t(\emptyplug)$:
this will be the end of $\Gamma_0$.
Notice that beginning and end have length at most $d$ each.
The middle of $\Gamma_0$ connects the final point of the first arc
to the initial point of the second arc along the spine $\cC^{\bullet}_{\cD}$:
from Lemma \ref{lemma:quasi},
it is the shortest arc connecting these two points.
Let $N_0$ be the length of $\Gamma_0$.

Given a tiling $\bt_1 \in \cT(\cR_{N_1})$, $\Tw(\bt) = t$,
consider the corresponding tiling path
$\Gamma_1: [0,{N_1}] \to \tilde\cC_{\cD}$
with endpoints $\Gamma_0(1) = \emptyplug$
and $\Gamma_1({N_1}) = \sigma^t(\emptyplug)$.
From Lemma \ref{lemma:quasi} and the triangle inequality,
${N_1} \ge N_0 - 4d$.
We construct a homotopy from $\Gamma_0$ to $\Gamma_1$.

We first define intermediate paths $\Gamma_{\frac{s}{1+N_1}}$
for $s \in \{1,2,\ldots,N_1\}$ as follows.
First follow $\Gamma_1$ from $\Gamma_1(0)$ to $\Gamma_1(s)$:
call this part the first part of $\Gamma_{\frac{s}{1+N_1}}$.
We now construct the second part of $\Gamma_{\frac{s}{1+N_1}}$
just as we constructed $\Gamma_0$.
More precisely: first construct the shortest arc 
from $\Gamma_1(s)$ to the spine $\cC^{\bullet}_{\cD}$:
this will be the beginning of the second part of $\Gamma_{\frac{s}{1+N_1}}$.
Next construct the shortest arc from
the spine $\cC^{\bullet}_{\cD}$ to $\sigma^t(\emptyplug)$:
this will be the end of the second part of $\Gamma_{\frac{s}{1+N_1}}$
(notice that it coincides with the end of $\Gamma_0$).
The middle of the second part of $\Gamma_{\frac{s}{1+N_1}}$
again connects the final point of the first arc
to the initial point of the second arc along the spine $\cC^{\bullet}_{\cD}$.
As above, the length of $\Gamma_{\frac{s}{1+N_1}}$ is at most $N_1 + 4d$.

We now need homotopies from
$\Gamma_{\frac{s}{1+N_1}}$ to
$\Gamma_{\frac{s+1}{1+N_1}}$.
The case $s = N_1$ is easy:
$\Gamma_{\frac{N_1}{1+N_1}}$ differs from $\Gamma_1$
just by the fact that at the end we add a path from
$\sigma^t(\emptyplug)$ to the spine $\cC^{\bullet}_{\cD}$ and back.
We may assume that it is the same path,
so the homotopy consists of using the second move above
to eliminate the difference
(or perhaps the reader prefers to imitate
the proof of Lemma \ref{lemma:eventiling}).
We thus focus in the case $s < N_1$.

Notice that 
$\Gamma_{\frac{s}{1+N_1}}$ and
$\Gamma_{\frac{s+1}{1+N_1}}$ coincide from $0$ to $s$.
The path $\Gamma_{\frac{s}{1+N_1}}$
then follows an arc from $\Gamma_1(s)$ to $\cC^{\bullet}_{\cD}$
and then follows along the spine $\cC^{\bullet}_{\cD}$.
The path $\Gamma_{\frac{s+1}{1+N_1}}$ 
first moves from $\Gamma_1(s)$ to $\Gamma_1(s+1)$,
then follows an arc from $\Gamma_1(s+1)$ to $\cC^{\bullet}_{\cD}$
and then follows along the spine $\cC^{\bullet}_{\cD}$
(see Figure \ref{fig:homotopy}).

\begin{figure}[ht]
\centering
\def\svgwidth{75mm}
\begingroup%
  \makeatletter%
  \providecommand\color[2][]{%
    \errmessage{(Inkscape) Color is used for the text in Inkscape, but the package 'color.sty' is not loaded}%
    \renewcommand\color[2][]{}%
  }%
  \providecommand\transparent[1]{%
    \errmessage{(Inkscape) Transparency is used (non-zero) for the text in Inkscape, but the package 'transparent.sty' is not loaded}%
    \renewcommand\transparent[1]{}%
  }%
  \providecommand\rotatebox[2]{#2}%
  \newcommand*\fsize{\dimexpr\f@size pt\relax}%
  \newcommand*\lineheight[1]{\fontsize{\fsize}{#1\fsize}\selectfont}%
  \ifx\svgwidth\undefined%
    \setlength{\unitlength}{453.62199457bp}%
    \ifx\svgscale\undefined%
      \relax%
    \else%
      \setlength{\unitlength}{\unitlength * \real{\svgscale}}%
    \fi%
  \else%
    \setlength{\unitlength}{\svgwidth}%
  \fi%
  \global\let\svgwidth\undefined%
  \global\let\svgscale\undefined%
  \makeatother%
  \begin{picture}(1,0.52927213)%
    \lineheight{1}%
    \setlength\tabcolsep{0pt}%
    \put(0,0){\includegraphics[width=\unitlength,page=1]{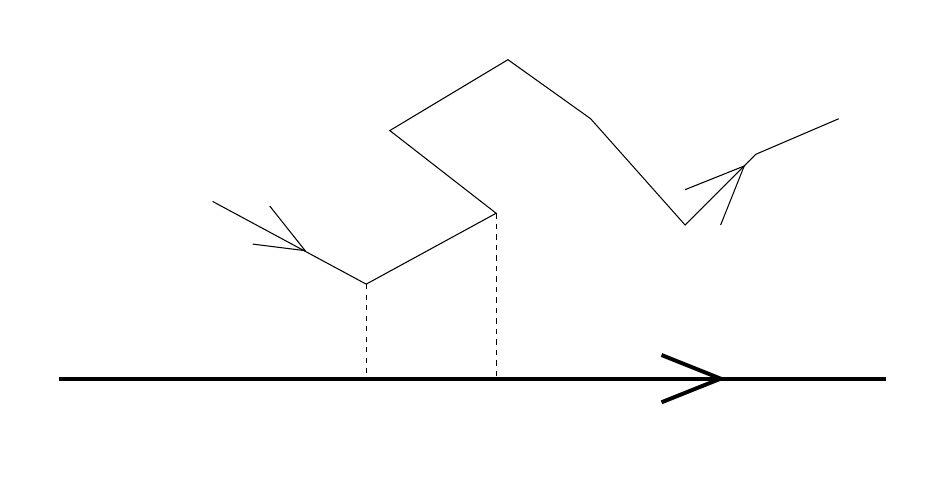}}%
    \put(0.75000619,0.06610925){\color[rgb]{0,0,0}\makebox(0,0)[lt]{\lineheight{1.25}\smash{\begin{tabular}[t]{l}$\cC^{\bullet}_{\cD}$\end{tabular}}}}%
    \put(0.24999381,0.35361637){\color[rgb]{0,0,0}\makebox(0,0)[lt]{\lineheight{1.25}\smash{\begin{tabular}[t]{l}$\Gamma_1$\end{tabular}}}}%
    \put(0.37569136,0.26403082){\color[rgb]{0,0,0}\makebox(0,0)[lt]{\lineheight{1.25}\smash{\begin{tabular}[t]{l}$s$\end{tabular}}}}%
    \put(0.54236216,0.2897259){\color[rgb]{0,0,0}\makebox(0,0)[lt]{\lineheight{1.25}\smash{\begin{tabular}[t]{l}$s+1$\end{tabular}}}}%
  \end{picture}%
\endgroup%

\caption{The two paths $\Gamma_{\frac{s}{1+N_1}}$ and
$\Gamma_{\frac{s+1}{1+N_1}}$ differ in a short arc only.}
\label{fig:homotopy}
\end{figure}

There exist therefore intervals of length at most $4d$ each
in the domains of $\Gamma_{\frac{s}{1+N_1}}$ and
$\Gamma_{\frac{s+1}{1+N_1}}$ such that these two paths
coincide outside the intervals.
We construct a homotopy from one to the other
changing these arcs only.
By definition of $\tilde M$,
such a homotopy exists using intermediate curves of length at most $\tilde M$.
Thus, when we plug this smaller homotopy inside the larger one
all intermediate curves have length at most $N_1 + M$,
as desired.
\end{proof}

\begin{coro}
\label{coro:giant}
Let $\cD$ be a regular quadriculated disk;
let $M \in \NN^{\ast}$ be as in Theorem \ref{theo:M}.
For $N \in \NN^{\ast}$ and $t \in \ZZ$,
let $\cT_{N,t} = \{ \bt \in \cT(\cR_N) \;|\; \Tw(\bt) = t \}$.
Partition $\cT_{N,t}$ by the equivalence relation $\approx$.
All tilings $\bt \in \cT_{N,t}$ having {\em at least} $M$
vertical floors belong to the same connected component.
\end{coro}

\begin{proof}
Let $\bt_0, \bt_1 \in \cT_{N,t}$ be two such tilings.
Apply Lemma \ref{lemma:movevert} to move vertical floors to the end
and therefore obtain $\tilde\bt_0, \tilde\bt_1 \in \cT(\cR_{N-M})$
with $\bt_i \approx \tilde\bt_1 \ast \bt_{\nvert,M}$ (for $i \in \{0,1\}$).
From Theorem \ref{theo:M},
$\tilde\bt_0 \ast \bt_{\nvert,M} \approx \tilde\bt_1 \ast \bt_{\nvert,M}$,
completing the proof.
\end{proof}

\begin{remark}
\label{remark:diameter}
It follows from the proof of Theorem \ref{theo:M} that
there exists a linear bound (as a function of $N$)
to the number of flips necessary to move in $\cT(\cR_{N+M})$
from $\bt_0 \ast \bt_{\nvert,M}$ to $\bt_1 \ast \bt_{\nvert,M}$,
where $\bt_0, \bt_1 \in \cT(\cR_N)$, $\Tw(\bt_0) = \Tw(\bt_1)$.
Indeed, it suffices to {\em first} establish the length of the longest path.
The contruction of $\Gamma_{\frac{s}{1+N_1}}$ is then done 
in the order above, with any extra length taken up 
by going back and forth along $\cC^{\bullet}_{\cD}$
at the start of the middle of the second part of $\Gamma_{\frac{s}{1+N_1}}$.
\end{remark}

\begin{remark}
\label{remark:hyperbolic}
The crucial property of the domino group $G_{\cD} = \pi_1(\cC_{\cD})$
in the proof of Theorem \ref{theo:M} appears to be hyperbolicity
(in the sense of Gromov, see \cite{gromov}).

The group $\ZZ$ is of course a rather too special example
of a hyperbolic group.
Work in progress indicates that for some irregular regions
the domino group $G_{\cD}$ has exponential growth
but also contains copies of $\ZZ^2$
and is therefore not hyperbolic.
\end{remark}


\section{Final remarks}
\label{sect:final}

The reader probably sees that many questions were left unanswered;
we make a few remarks about them.

The most obvious question is probably:
exactly which quadriculated disks are regular?
This question is briefly discussed at the end of Section \ref{sect:44}.
As mentioned there, we do not have a complete answer.
Significant progress has been acheived in \cite{marreiros},
and there is more related work in progress.


Another question is to compute the domino group $G_{\cD}$
for examples of non regular quadriculated disks $\cD$.
Notice that even for rectangles $[0,2] \times [0,N]$
we did not compute the group;
in fact it is not hard to see (particularly using \cite{primeiroartigo})
that the map we constructed is not an isomorphism.
Also here, there is work in progress.

Theorem \ref{theo:M} invites several questions.
We might want to determine the best value of $M$ as a function of $\cD$.
Computations for the example $\cD = [0,4]^2$
show that if $N = 4$ then $M = 2$ works;
there is some evidence suggesting that $M = 2$ may work for all $N$.
At this point, it is even consistent with what we know
that $M$ can be taken as a constant independent of $\cD$;
perhaps even $M = 2$ works.
It would also be interesting to obtain a similar result
without the hypothesis of $\cD$ being regular.
The general case (when $\cD$ is not regular)
probably depends on the structure of the domino group $G_{\cD}$;
as discussed in Remark \ref{remark:hyperbolic},
the concept of hyperbolicity may be relevant.

In \cite{saldanhaejc} we present some probabilistic results.
In particular, we use Theorem \ref{theo:M} and Corollary \ref{coro:giant}
(together with an estimate on the distribution of twist)
to say something about the number and sizes 
of the connected components (or equivalence classes) under $\approx$.
In particular, we see that small values of twist
account for almost all tilings and that,
given a small value for the twist,
the connected component described in Corollary \ref{coro:giant}
is a ``giant component''
which contains almost all tilings with that twist.
Still, it would be interesting to have a better understanding
of the smaller connected components.

Another natural question is:
what happens in dimensions $4$ and higher?
This is the subject of \cite{KS}:
the twist can be defined, but now with values in $\ZZ/(2)$.
Many regions are regular;
for regular regions, if two tilings have the same twist
then almost always they can be joined by a finite sequence of flips.
Also, if a little extra space is allowed,
then two tilings with the same twist can always be joined by flips.
Thus, the space of tilings has two twin giant components,
one for each value of the twist.


\medskip

\noindent
\footnotesize
Departamento de Matem\'atica, PUC-Rio \\
Rua Marqu\^es de S\~ao Vicente, 225, Rio de Janeiro, RJ 22451-900, Brazil \\
\url{saldanha@puc-rio.br}

\end{document}